\documentclass[12pt,a4paper]{amsart}
\makeatletter
\usepackage[english]{babel}
\usepackage[utf8]{inputenc}
\usepackage[square,sort,comma,numbers]{natbib}
\usepackage{graphicx}
\usepackage{amsfonts}
\usepackage{amsthm}
\usepackage{sansmath}
\usepackage{hyperref}

\usepackage{fancyhdr}
\usepackage{amsmath,amssymb,euscript}
\usepackage{xparse}
\usepackage{tikz-cd}
\usepackage{verbatim}
\usepackage{mathtools}
\usepackage[new]{old-arrows}
\setcitestyle{square}

\numberwithin{equation}{section}

\theoremstyle{plain}
	\newtheorem{theo}[equation]{Theorem}
	\newtheorem{prop}[equation]{Proposition}
	
	\newtheorem{lemm}[equation]{Lemma}

	\newtheorem{corr}[equation]{Corollary}

	\newtheorem{lem/defn}[equation]{Lemma/Definition}

\theoremstyle{definition}
	\newtheorem{defi}[equation]{Definition}
	\newtheorem*{ackn}{Acknowledgements}

	\newtheorem{cons}[equation]{Construction}
	
	\newtheorem{nota}[equation]{Notation}

\theoremstyle{remark}
	\newtheorem{rema}[equation]{Remark}

\def\nc{\newcommand}
\def\on{\operatorname}
\def\co{\colon\thinspace}

\newcommand{\RomanNumeralCaps}[1]
    {\MakeUppercase{\romannumeral #1}}
\newcommand*{\rom}[1]{\expandafter\@slowromancap\romannumeral #1@}

\nc{\edit}[1]{\marginpar{\footnotesize{#1}}}
\newcommand{\lv}{\lvert}
\newcommand{\rv}{\rvert}
\newcommand{\ff}{\mathit{ff}}
\nc{\Z}{\mathbb{Z}}
\nc{\z}{\mathbb{Z}}
\nc{\zp}{\mathbb{Z}_p}
\nc{\zpl}{\mathbb{Z}_{(p)}}
\nc{\PP}{\mathbb{P}}
\nc{\R}{\mathbb{R}}
\nc{\ot}{\otimes}
\nc{\I}{\mathbb{I}}
\nc{\f}{\mathbb{F}}
\nc{\fp}{\mathbb{F}_p}
\nc{\hz}{H\mathbb{Z}}
\nc{\hzp}{H\mathbb{Z}_p}
\nc{\hzpl}{H\mathbb{Z}_{(p)}}
\nc{\fq}{\mathbb{F}_q}
\nc{\fpn}{\mathbb{F}_{p^n}}

\nc{\wfq}{W(\mathbb{F}_q)}
\nc{\wfpn}{W(\mathbb{F}_{p^n})}

\nc{\pmot}{\frac{p-1}{2}}

\nc{\sph}{\mathbb{S}}

\nc{\sphtriv}{\mathbb{S}^{\on{triv}}}
\nc{\sphp}{\mathbb{S}_{p}}
\nc{\sphpl}{\mathbb{S}_{(p)}}

\nc{\sphwq}{\mathbb{S}_{W(\mathbb{F}_q)}}
\nc{\sphwpn}{\mathbb{S}_{W(\mathbb{F}_{p^n})}}
\nc{\sphwk}{\mathbb{S}_{W(k)}}
\nc{\sphxk}{\sph[x_k]}
\nc{\sphxm}{\sph[x_{m}]}
\nc{\sphxmk}{\sph[x_{mk}]}
\nc{\sphsk}{\sph[\sigma_k]}
\nc{\sphsmk}{\sph[\sigma_{mk}]}
\nc{\sphsell}{\sph[\sigma_\ell]}
\nc{\sphpsk}{\sph_{(p)}[\sigma_k]}
\nc{\sphpzk}{\sph_{(p)}[z_k]}
\nc{\sphpzpmone}{\sph_{(p)}[z_{2(p-1)}]}
\nc{\sphpztwo}{\sph_{(p)}[z_2]}
\nc{\sphpsmk}{\sph_{(p)}[\sigma_{mk}]}
\nc{\sphpsell}{\sph_{(p)}[\sigma_\ell]}

\nc{\sphzmk}{\sph[z_{mk}]}
\nc{\sphzk}{\sph[z_k]}
\nc{\sphzpmone}{\sph[z_{2(p-1)}]}
\nc{\sphztwo}{\sph[z_2]}

\nc{\musgmpn}{MU[\sigma_{2p^n-2}]}
\nc{\muallspi}{MU[\sigma_{2p^i-2} \mid i< n]}
\nc{\muallspinpo}{MU[\sigma_{2p^i-2} \mid i< n+1]}
\nc{\wdgmuallspi}{\wedge_{MU[\sigma_{2p^i-2} \mid i< n]}}
\nc{\thhmuallspi}{\ensuremath{\textup{THH}^{MU[\sigma_{2p^i-2} \mid i<n]}}}
\nc{\musi}{MU[\sigma_i]}
\nc{\musk}{MU[\sigma_k]}
\nc{\muspi}{MU[\sigma_{2p^i-2}]}

\nc{\sgmk}{\sigma_k}
\nc{\sgmmk}{\sigma_{mk}}
\nc{\sgml}{\sigma_{\ell}}
\nc{\sgmpi}{\sigma_{2p^i-2}}

\nc{\xrtmx}{X(\sqrt[m]{x})}
\nc{\artma}{A(\sqrt[m]{a})}

\nc{\hzplsk}{H\mathbb{Z}_{(p)}[\sigma_{k}]}

\nc{\hzplsmk}{H\mathbb{Z}_{(p)}[\sigma_{mk}]}

\nc{\wdgsphsmk}{\wdg_{\sph[\sigma_{mk}]}}
\nc{\wdgsphsk}{\wdg_{\sph[\sigma_{k}]}}
\nc{\wdgsphpsmk}{\wdg_{\sph_{(p)}[\sigma_{mk}]}}
\nc{\wdgsphpsk}{\wdg_{\sph_{(p)}[\sigma_{k}]}}

\nc{\bpn}{BP \langle n \rangle}
\nc{\hfp}{H\mathbb{F}_p}
\nc{\T}{\mathbb{T}}
\nc{\vo}{V(1)}
\nc{\vos}{V(1)_*}
\nc{\ttw}{T(2)}
\nc{\pis}{\pi_*}
\nc{\ttws}{T(2)_*}
\nc{\gdna}{\gdn(A\hmod)}

\nc{\gdnr}{\gdn(R\hmod)}

\nc{\wdg}{\wedge}

\nc{\wdgp}{\wedge_{\mathbb{S}_p}}
\nc{\wdgfp}{\wedge_{H\mathbb{F}_p}}
\nc{\wdgmu}{\wedge_{MU}}

\nc{\AAA}{\mathbb{A}}
\nc{\LL}{\mathbb{L}}
\nc{\OO}{\mathcal{O}}

\nc{\X}{\EuScript{X}}
\nc{\sZ}{\EuScript{Z}}

\nc{\id}{{\on{id}}}
\nc\cone{{\on{cone}}}
\nc{\Rep}{{\on{Rep}}}
\nc\Ob{{\on{Ob}}}
\nc\Spec{{\on{Spec}}}

\mathchardef\mhyphen="2D
\newcommand{\hmod}{\mhyphen\mathsf{Mod}}
\nc\coMod{{\on{coMod}}}
\nc\Perf{{\on{Perf}}}
\nc\End{{\on{End}}}
\nc{\into}{\hookrightarrow}
\nc{\tr}{\on{tr}}
\nc{\ev}{\on{ev}}
\nc{\im}{\on{im}}
\nc{\hfps}{H{\mathbb{F}_p}_*}
\nc{\Mot}{\on{Mot}}
\nc{\pt}{\on{pt}}
\nc{\coker}{\on{coker}}
\nc{\rk}{\on{rank}}
\nc{\TOP}{\on{Top}_{\mathbb{C}}^{s}}
\nc{\gr}{\on{Gr}}
\nc{\Catperf}{\text{Cat}^{\text{perf}}}
\nc{\Sym}{\on{Sym}}
\nc{\xra}{\xrightarrow}
\nc{\lra}{\xleftarrow}
\nc{\Bet}{\mathbf{Betti}_{X}}
\nc{\codim}{\on{codim}}
\nc{\Fred}{\on{Fred}}
\nc{\colim}{\on{colim}}
\nc{\KK}{{\bf K}}

\nc{\onto}{\twoheadrightarrow}
\nc{\A}{\mathbb{A}}
\nc{\Aff}{\on{Aff}}
\nc{\SH}{\on{SH}}
\nc{\QCoh}{\on{QCoh}}
\nc{\Alg}{\on{Alg}}
\nc{\alg}{\on{Alg}}
\nc{\lmod}{\on{LMod}}

\nc{\rmod}{\on{RMod}}
\nc{\Br}{\on{Br}}
\nc{\ta}{\widetilde{\a}}
\nc{\Shv}{\on{Shv}}
\nc{\GG}{\mathbb{G}}
\nc{\red}{\color{red}}
\nc{\blue}{\color{blue}}
\nc{\an}{\on{an}}
\nc{\Pre}{\on{Pre}}
\nc{\assact}{\on{Ass}_{\on{act}}^{\otimes }}
\nc{\spact}{\on{Sp}_{\on{act}}^{\otimes }}
\nc{\spactprod}{{\on{Sp}^{\Z}_{\on{act}}}^{\otimes }}
\nc{\qc}{\on{qc}}
\nc{\op}{\on{op}}
\nc{\shEnd}{{\mathcal End}}
\nc{\Sph}{\mathbb{S}}
\nc{\Top}{\on{Top}}
\nc{\Map}{\on{Map}}
\nc{\Vect}{\on{Vect}}
\nc{\holim}{\on{holim}}
\nc{\fun}{\on{Fun}}

\nc{\leqgrsp}{\ensuremath{\textup{Leq}}\big(\grm(\sp^{BS^1}),(-)^{tC_p}\big)}
\newcommand{\cat}[1]{\ensuremath{\EuScript #1}}

\DeclareMathOperator{\map}{\ensuremath{\textup{Map}}}

\newcommand{\spcircle}{\ensuremath{\textup{Sp}^{BS^1}}}

\DeclareMathOperator{\triv}{\ensuremath{\textup{triv}}}

\DeclareMathOperator{\thh}{\ensuremath{\textup{THH}}}
\newcommand{\logthhku}{\thh(ku_{(p)}\mid u_2)}

\newcommand{\logthhell}{\thh(\ell \mid v_1)}

\DeclareMathOperator{\grm}{\ensuremath{\textup{Gr}}_m}
\DeclareMathOperator{\grmsp}{\ensuremath{\textup{Gr}}_m(\sp)}

\DeclareMathOperator{\ntc}{\ensuremath{\textup{TC}^-}}
\DeclareMathOperator{\tp}{\ensuremath{\textup{TP}}}
\DeclareMathOperator{\tc}{\ensuremath{\textup{TC}}}

\DeclareMathOperator{\kth}{\ensuremath{\textup{K}}}
\newcommand{\kup}{ku_p}
\newcommand{\kupl}{ku_{(p)}}

\nc{\C}{\cat C}
\nc{\D}{\cat D}
\nc{\V}{\cat V}

\def\A{\mathcal{A}}

\def\a{\alpha}

\def\Perf{\on{Perf}}

\def\sp{\on{Sp}}

\def\cycsp{\on{CycSp}}

\def\sp{\on{Sp}}

\def\V{\EuScript{V}}

\nc{\W}{\mathbb{W}}

\def\QCoh{\on{QCoh}}

\title{Algebraic $K$-theory of the two-periodic first Morava $K$-theory} 
\author{Haldun \"Ozg\"ur Bay{\i}nd{\i}r}
\usepackage{natbib}
\usepackage{graphicx}

\begin{document}
\maketitle
\begin{abstract}
Using the root adjunction formalism developed in an earlier work and logarithmic THH, we obtain a simplified computation of $T(2)_*\text{K}(ku)$ for $p>3$. Through this, we also produce a new algebraic $K$-theory computation; namely  we obtain   $T(2)_*\text{K}(ku/p)$, where $ku/p$ is  the $2$-periodic Morava $K$-theory spectrum of height $1$.
\end{abstract}

\section{Introduction}

\let\thefootnote\relax\footnotetext{2020 \textit{Mathematics Subject Classification.} Primary  19D55, 55P43, 55Q51}
In this work, we continue the study from \cite{ausoni2022adjroot}  
of the effect of root adjunctions of ring spectra at the level of algebraic $K$-theory. In particular, we consider the maps of ring spectra $\ell_p \to ku_p$ and $k(1)\to ku/p$ where $\ell_p$ denotes the Adams summand of the connective $p$-completed complex $K$-theory spectrum $ku_p$ and $k(1)$ denotes the connective first Morava $K$-theory spectrum. 

Ausoni and Rognes compute the $V(1)$-homotopy of $\kth(\ell_p)$ in \cite{ausonirogneskthryoftopologicalkthry} for $p>3$.  Later, Ausoni extends this to a computation of $V(1)_* \kth(ku_p)$  \cite{ausoni2010kthryofcomplexkthry}. An interest for $\kth(ku)$ stems from the fact that it classifies virtual 2-vector bundles, a 2-categorical analogue of ordinary
complex vector bundles \cite{baas2011stablebundlesrig}. 

As an outcome of his computations, Ausoni observes that the relationship between $V(1)_*\kth(\ell_p)$ and $V(1)_*\kth(ku_p)$ through the map $V(1)_*\kth(\ell_p) \to V(1)_*\kth(ku_p)$ resembles a height $2$ analogue of $\kth_*(\zp;\z/p) \to \kth_*(\zp[\zeta_p];\z/p)$ for the cyclotomic extension $\z_p\to \zp[\zeta_p]$ where $\zeta_p$ is a primitive $p$th root of unity; the computation of the former is due to  Hesselholt and Madsen \cite[Theorem D]{hesselholt2003ktheoryoflocalfields}. For instance, $\kth(\zp[\zeta_p];\z/p)$ is essentially given by adjoining a $p-1$-root to $v_1$ in  $\kth(\zp;\z/p)$. For  $V(1)_*\kth(ku_p)$, Ausoni shows that there is a $p-1$-root $b$ of $-v_2$, i.e.\ $b^{p-1} = -v_2$. Furthermore, he proves that $T(2)_*\kth(ku_p)$ is given by adjoining a $p-1$-root to $-v_2$ in $T(2)_*\kth(\ell_p)$ as follows.

\begin{theo}[\cite{ausoni2010kthryofcomplexkthry}, Theorem \ref{theo restated Ausoni algebraic k theory of ku}]\label{theo Ausoni algebraic k theory of ku}
Let $p>3$ be a prime. There is an isomorphism of $\fp[b]$-algebras:
\[T(2)_* \kth(ku) \cong T(2)_* \kth(\ell) [b]/(b^{p-1}+v_2).\]
where $\lv b \rv = 2p+2$.
\end{theo}

\begin{rema}
Since $T(2)_*\kth(\ell)$ is known due to Ausoni and Rognes \cite[Theorem 0.3]{ausonirogneskthryoftopologicalkthry}, the theorem above provides an explicit description of $T(2)_*\kth(ku)$.
\end{rema}

Following this comparison, Ausoni, the author and Moulinos construct a root adjunction method for ring spectra and study the algebraic $K$-theory, THH and logarithmic THH of ring spectra obtained via root adjunction  \cite{ausoni2022adjroot}. Let $A$ be an $E_1$-ring spectrum and let $a \in \pi_{mk}A$. Under suitable hypothesis, this construction provides another $E_1$-ring  $A(\sqrt[m]{a})$ for which the homotopy ring  of $A(\sqrt[m]{a})$ is precisely given by a root adjunction:
\[\pis A(\sqrt[m]{a}) \cong \pis A [z]/(z^m-a).\]

Furthermore, $A(\sqrt[m]{a})$ is an  $E_1$-algebra in  $\fun(\z/(m)^{\textup{ds}}, \sp)$ equipped with the Day convolution symmetric monoidal structure; we say $A(\sqrt[m]{a})$ is an $m$-graded $E_1$-ring. Roughly speaking, this structure may be considered as  a splitting $A(\sqrt[m]{a})\simeq \vee_{i \in \z/m} A(\sqrt[m]{a})_{i}$, which we call the weight grading on $A(\sqrt[m]{a})$, for which  the multiplication on $A(\sqrt[m]{a})$ is given by maps respecting this grading over $\z/m$:
\[A(\sqrt[m]{a})_i \wdg  A(\sqrt[m]{a})_j \to A(\sqrt[m]{a})_{i+j}.\]
 Furthermore, we have $A(\sqrt[m]{a})_i = \Sigma^{ik}A$ for $0\leq i<m$. This results in a canonical splitting of $\thh(A\sqrt[m]{a})$ into a coproduct of $m$-cofactors as an $S^1$-equivariant spectrum. It follows by \cite[Theorem 1.9]{ausoni2022adjroot} that at the level of algebraic $K$-theory,  the  map \[\kth(A)\to \kth(A(\sqrt[m]{a}))\] is the inclusion of a wedge summand whenever $A$ is $p$-local and $p \nmid m$. 

The authors prove in \cite{ausoni2022adjroot} that there is an equivalence of $E_1$-rings $ku_p \simeq \ell_p(\sqrt[p-1]{v_1})$. This equips $ku_p$ with the structure of a $p-1$-graded $E_1$-ring through $ku_p \simeq \vee_{0\leq i<p-1} \Sigma^{2i}\ell_p$ and one obtains that $\thh(ku_p)$  admits  an $S^1$-equivariant splittings into $p-1$ summands. 
\begin{rema}
    Results of similar nature are given in \cite{carmeli2021chromatic} and \cite{ben2023descent}. In \cite{carmeli2021chromatic}, Carmeli, Schlank and Yanovski define a notion of cyclotomic extensions for $T(n)$-local (or $K(n)$-local) $E_\infty$-ring spectra and in \cite{ben2023descent}, the same authors and Ben-Moshe prove that $T(n+1)$-local  algebraic $K$-theory carries height $n$ cyclotomic extensions to height $n+1$ cyclotomic extensions. These results may be considered to be analogous to Theorem \ref{theo Ausoni algebraic k theory of ku}. 
\end{rema}
In this work, our first main objective is to make a simplified computation of $T(2)_*\kth(ku)$ which allows for a clear and conceptual picture of how Theorem \ref{theo Ausoni algebraic k theory of ku} is proved. For this, we begin by showing that $ku_p$ admits the structure of a $p-1$-graded $E_\infty$-ring lifting its $p-1$-graded $E_1$-ring structure  provided by root adjunction. Moreover, we show  that  this equips $\thh(\kup)$ and $\tc(ku_p)$ with $p-1$-graded $E_\infty$-ring structures and  results in an $S^1$-equivariant splitting of the logarithmic THH of $ku_p$ (in the sense of Rognes \cite{rognes2009topologicallogarithmic} computed in \cite{rognes2018logthhofku}). After this, we obtain Theorem \ref{theo Ausoni algebraic k theory of ku} through a computation involving logarithmic THH of $ku_p$. In \cite[Sections 3 and 4]{ausoni2010kthryofcomplexkthry}, Ausoni constructs what he calls the higher Bott element $b \in V(1)_{2p+2}\kth(ku_p)$ and he identifies the image of $b$ in $V(1)_*\thh(ku_p)$. This is the only result that we take from Ausoni's work as input. In particular, our computation of $T(2)_*\kth(ku)$ avoids the low dimensional computations and the infinite spectral sequence argument of \cite[Sections 5, 6 and 7]{ausoni2010kthryofcomplexkthry}. We give an outline of our computation in Section \ref{sec outline} below.

\begin{rema}
    In \cite{hahn2022motivic}, the authors construct a motivic filtration for the topological cyclicic homology of ring spectra and using this, they obtain a simplified computation of $V(1)_*\tc(\ell_p)$ which also applies to the case $p=3$. In this work, we do not use the motivic filtration of \cite{hahn2022motivic}. On the other hand, the motivic filtration should make it possible to obtain another simplified computation of $T(2)_*\kth(ku)$ which also recovers $V(1)_*\kth(ku_p)$ including the case of $p=3$. We hope that our methods, including the use of the $p-1$-graded ring structure on $ku_p$, would also be useful for the motivic approach to $V(1)_*\kth(ku_p)$.
    \end{rema}
\begin{rema}\label{rema iterated kthry of ku}
The iterated algebraic $K$-theory spectrum  $K^{(n)}(ku)$ is expected have an interpretation in terms of a  suitable notion of $n+1$-vector bundles \cite{lind2020twistediteratedkthryofku}. We believe that our approach to $\tc(ku_p)$ and the $p-1$-graded $E_\infty$-ring structure we construct on $\tc(ku_p)$ will be useful for the study of $K^{(n)}(ku)$.
\end{rema}
\begin{rema}\label{rema sagave lundemo and ausoni}
Currently, Christian Ausoni, the author, Tommy Lundemo and Steffen Sagave are working on generalizing the methods of this work  to obtain a higher height analogue of Theorem \ref{theo Ausoni algebraic k theory of ku} that relates the algebraic $K$-theory of $E_n$ to that of the truncated Brown-Peterson spectrum $\bpn$.
\end{rema}

In a later work \cite{ausoni2012kthrymoravaktheory}, Ausoni and Rognes compute $V(1)_*\kth(\ell/p)$ where $\ell/p$ is the connective Morava $K$-theory spectrum $k(1)$ of height one. Our computational approach to $T(2)_*\kth(ku_p)$ naturally provides a computation of $T(2)_*\kth(ku/p)$; see Section \ref{sec outline} for an outline. Namely, we obtain the first computation of $T(2)_*\kth(ku/p)$. Note that, $ku/p$ is also called the connective $2$-periodic Morava $K$-theory of height one. For the following, we consider $T(2)_*\kth(ku/p)$ as an $\fp[b]$-module through its  $T(2)_*\kth(ku_p)$-module structure; $b$ is the higher Bott element as before. 
\begin{theo}[Theorem \ref{theo restated kth of ku mod p}]\label{theo kth of ku mod p}
Let $p>3$ be a prime. There is an isomorphism of $\fp[b]$-modules:
\[T(2)_*\kth(ku/p) \cong T(2)_* \kth(\ell/p)\otimes_{\fp[v_2]} \fp[b]\]
with $\lv b \rv = 2p+2$ and in the tensor product above, we take $v_2 = - b^{p-1}$. 
\end{theo}

Together with \cite[Theorem 1.1]{ausoni2012kthrymoravaktheory}, the theorem above provides a complete description of $T(2)_*\kth(ku/p)$. In particular, similar to Theorem \ref{theo Ausoni algebraic k theory of ku}, $T(2)_*\kth(ku/p)$ is given by a $p-1$-fold coproduct of shifted  copies of $T(2)_*\kth(\ell/p)$.

Indeed, Ausoni and Rognes compute $V(1)_*\kth(\ell/p)$ with the goal of investigating
how localization and Galois descent techniques that have been used for (local)
number rings or fields can be applied to the algebraic K-theory of ring-spectra,
in particular $\ell_p$ and $ku_p$. For instance, there is a localization sequence 
\[\kth(\z/p) \to \kth(\z_p) \to \kth(\mathbb{Q}_p)\]
and the $T(1)$-homology of the algebraic $K$-theory of the fraction field $\mathbb{Q}_p$ is the target of a Galois descent spectra sequence \cite{thomason1985algkthryandetalecohomology}. 

To carry this picture to $\ell_p$, they define $\kth(\ff (\ell_p))$, what they call the algebraic $K$-theory of the fraction field of $\ell_p$, as the cofiber of the transfer map  $\kth(L/p) \to \kth(L_p)$.  Note that $\kth(\ff(\ell_p))$ is not claimed to be the algebraic $K$-theory of an $E_1$-ring. After computing  $T(2)_* \kth (\ff (\ell_p))$ \cite{ausoni2009kthryoffractionfield}, Ausoni and Rognes conjecture that this should also be the target of a Galois descent spectral sequence as their computation suggest. Carrying this discussion to $ku_p$, they state a conjectural formula \cite[Section 3]{ausoni2009kthryoffractionfield}:
\begin{equation}\label{eq conjectural formula of ausoni rognes}
    T(2)_*\kth(\ff(ku_p)) \cong T(2)_*\kth(\ff(\ell_p)) \otimes_{ \fp[v_2]} \fp[b], 
\end{equation}
here, $\kth(\ff(ku_p))$ is defined to be the cofiber of the transfer map below. 
\[\kth(KU/p) \to \kth(KU_p) \to \kth(\ff(ku_p))\]
Using Theorem \ref{theo kth of ku mod p}, we prove the conjectural formula in  \eqref{eq conjectural formula of ausoni rognes}, see Theorem \ref{theo k theory of the fraction field of ku}.

\begin{nota}
For $p>3$, we let $V(1)$ denote the spectrum given by $\sph/(p,v_1)$. Due to \cite{oka1982multiplicativefiniteringspectra}, this is a homotopy commutative ring spectrum. Inverting the self map $v_2$ of $V(1)$, we obtain $T(2):=V(1)[v_2^{\pm 1}]$. Since $L_{T(2)}V(1) \simeq T(2)$ \cite[Section 1.5]{hovey1997vnelementsinringspectraandbordism}, we obtain that $T(2)$ is also a homotopy commutative ring spectrum and  $V(1) \to T(2)$ is a map of  a homotopy commutative ring spectra.

We let $\cycsp$ denote the $\infty$-category of cyclotomic spectra as in \cite[Definition 2.1]{antieau2020beilinson}; this is a slight variation of what is called $p$-cyclotomic spectra  in \cite{nikolausscholze2018topologicalcyclic}. In particular, an  object of $\cycsp$ is an $S^1$-equivariant spectrum $E$ with an $S^1$-equivariant map $E \to E^{tC_p}$.

 For a given small symmetric monoidal $\infty$-category $\C$ and a presentably symmetric monoidal $\infty$-category $\cat D$, we let $\fun(\cat C, \cat D)$ denote the corresponding functor $\infty$-category equipped with the symmetric monoidal structure given by Day convolution \cite{glasman2016dayconv, day1970closedcategoriesoffunctors}. For a simplicial set $K$, we let $\D^{K}$  denote the symmetric monoidal $\infty$-category given by the simplicial set of maps from $K$ to $\D$ equipped with the levelwise symmetric monoidal structure \cite[Remark 2.1.3.4]{lurie2016higher}.
 
 For $m \geq 0$, we when we mention the abelian group $\z/m$ as a symmetric monoidal monoidal $\infty$-category, we mean the corresponding discrete symmetric monoidal $\infty$-category. This is often denoted by $(\z/m)^{\textup{ds}}$ in the literature.  
 
\end{nota}
\begin{ackn}
I would like to thank Christian Ausoni for introducing me to this subject and for the valuable discussions I have had with him that led to this project. I also would like to thank Tasos Moulinos and Maximillien Peroux for answering my various questions regarding this work. Furthermore, I also benefited from  discussions with Gabriel Angelini-Knoll, Andrew Baker,  Jeremy Hahn, Yonatan Harpaz, Eva H\"oning,  Thomas Nikolaus, Tommy Lundemo, Birgit Richter and Steffen Sagave; I would like to thank them as well. We  acknowledge support from the project ANR-16-CE40-0003 ChroK and the Engineering and Physical Sciences Research Council (EPSRC)  grant EP/T030771/1.
\end{ackn}

\section{Outline}\label{sec outline}

Here, we provide an outline of the proofs of Theorems \ref{theo Ausoni algebraic k theory of ku} and \ref{theo kth of ku mod p}. In  Section \ref{sec einfty grading on ku}, we use the $\sp$-linear Fourier transform developed in \cite{carmeli2021chromatic} to show that the action of the $p$-adic Adams operations on $ku_p$ equips $ku_p$ with the structure of a $p-1$-graded $E_\infty$-ring, i.e.\ an $E_\infty$-algebra in $\fun(\z/(p-1),\sp)$, compatible with the splitting \[ku_p \simeq \vee_{0\leq i<p-1}\Sigma^{2i}\ell_p.\] Furthermore, we show that the resulting $p-1$-graded $E_1$-ring structure on $ku_p$ agrees with that provided by the root adjunction methods of \cite{ausoni2022adjroot}, i.e.\ that provided by the equivalence $ku_p \simeq \ell_p(\sqrt[p-1]{v_1})$. 

In Section \ref{sec graded THH, TC, neg TC}, we discuss the resulting grading on $\thh(ku_p)$ and $\tc(ku_p)$ with further details provided in Appendix \ref{appendix graded tc} where we show that $\tc(-)$ is a lax symmetric monoidal functor from the $\infty$-category of $p-1$-graded $E_1$-rings to the $\infty$-category of $p-1$-graded spectra. We deduce that  $\thh(ku_p)$ is an $E_\infty$-algebra in $p-1$-graded cyclotomic spectra and that $\tc(ku_p)$ and $\ntc(ku_p)$ are $p-1$-graded $E_\infty$-rings. 

Since $ku_p \simeq \ell_p(\sqrt[p-1]{v_1})$, it follows from the results of \cite{ausoni2022adjroot} that \[\tc(ku_p)_0\simeq \tc(\ell_p) \textup{\ \ and \ }\ntc(ku_p)_0\simeq \ntc(\ell_p).\] Therefore, to obtain Theorem \ref{theo Ausoni algebraic k theory of ku}, it suffices to show that there is a unit 
\[b \in T(2)_{2p+2} \tc(ku_p)\]
of weight $1$. For this, we use the element $b \in V(1)_{2p+2} \kth(ku_p)$ constructed in \cite[Section 3]{ausoni2010kthryofcomplexkthry}. 

We show in Section \ref{sec log thh with an equivariant splitting} that logarithmic THH of $ku_p$ (as in \cite{rognes2018logthhofku}) also admits an $S^1$-equivariant splitting compatible with that on $\thh(ku_p)$. After this, we obtain that $b$ represents a unit in $T(2)_*\ntc(ku_p)$ by using the logarithmic THH computations of Rognes,  Sagave and Schlichtkrull \cite{rognes2018logthhofku}, see Section \ref{sec topological cyclic homology of kup and ku mod p}. From this, it follows easily that $b$ is indeed a unit of weight $1$ in $T(2)_*\tc(ku_p)$. This provides $T(2)_*\tc(ku_p)$, i.e.\  Theorem \ref{theo Ausoni algebraic k theory of ku}, since \[T(2)_* \tc(ku_p)_0 \cong T(2)_*\tc(\ell_p)\] and that $T(2)_*\tc(ku_p)$ is periodic in the weight direction due to the unit $b$ of weight $1$ in $T(2)_*\tc(ku_p)$.

To compute $T(2)_*\kth(ku/p)$, i.e.\ to prove Theorem \ref{theo kth of ku mod p}, we construct $ku/p$ as an algebra over $ku_p$ in the $\infty$-category of $p-1$-graded spectra in Section \ref{sec einfty grading on ku}, i.e.\ $ku/p$ is a $p-1$-graded $ku_p$-algebra. Furthermore, we show that \[ku/p \simeq \ell/p(\sqrt[p-1]{v_1})\] as $p-1$-graded $E_1$-rings. As a result, we obtain that $\tc(ku/p)$ is a $p-1$-graded $\tc(ku_p)$-module (i.e.\ a module over $\tc(ku_p)$ in $\fun(\z/(p-1),\sp)$) and that \[\tc(ku/p)_0\simeq \tc(\ell/p).\] After this, Theorem \ref{theo kth of ku mod p} follows by noting that $T(2)_*\tc(ku_p)$ contains a unit of weight $1$ and therefore, every $p-1$-graded module over it, such as $T(2)_* \tc(ku/p)$, is periodic in the weight direction, see Section \ref{sec topological cyclic homology of kup and ku mod p}. 

\section{Graded ring spectra}\label{sec graded ring spectra}

Here, we set our conventions for graded objects in a  presentably symmetric monoidal stable $\infty$-category $\C$. We start by noting that there is an equivalence of $\infty$-categories
\[\fun(\z/m,\C)\simeq \prod_{i \in \z/m} \C.\]
We call an object $C$ of $\fun(\z/m,\C)$ an $m$-graded object of $\C$ and let $C_i$ denote $C(i)$. If $C$ is an $E_k$-algebra in $\fun(\z/m,\C)$, we say $C$ is an $m$-graded $E_k$-algebra in $\C$. If $C'$ is an $E_{k-1}$ $C$-algebra in $\fun(\z/m,\C)$, we say $C'$ is an $m$-graded $C$-algebra. For an $M \in  \fun(\z/m,\sp)$, we say $M$ is an $m$-graded spectrum and an $E_k$-algebra in $\fun(\z/m,\sp)$ is called an $m$-graded  $E_k$-ring. For $m=0$, we drop $m$ and talk about graded spectra, graded $E_k$-rings and so on.

The map $\z/m\to 0$ provides a symmetric monoidal left adjoint functor
\[D \co \fun(\z/m, \C) \to \fun(0,\C)\simeq \C\]
given by left Kan extension \cite[Corollary 3.8]{nikolaus2016stablemultyoneda}. We call $D(C)$ the underlying object of $C$ and this is given by the formula 
\[D(M) \simeq \coprod_{i\in \z/m} M_i.\]
We often omit $D$ in our notation.

We say an $m$-graded object $C$ of $\C$ is concentrated in weight $0$ if $C_i\simeq 0$ whenever $i \neq 0$. The inclusion $0\to \z/m$ provides another  adjunction:
\begin{equation*} \label{diag restriction and kan extension}
 \begin{tikzcd}
 {  \C\simeq \on{Fun}(0,\C)  } \arrow[r,"F_0",shift left]
 & {  \on{Fun}(\z/m,\C)  }\arrow[l,"G_0", shift left] 
 \end{tikzcd}
 \end{equation*}
where the left adjoint $F_0$ is symmetric monoidal and given by left Kan extension and $G_0$ is given by restriction, i.e.\ $G_0(C) = C_0$. For  $C\in \C$, $F_0(C)$ provides $C$ as an $m$-graded object concentrated in weight $0$. We often omit $F_0$ and for a given $C \in \fun(\z/m,\C)$, we denote the $m$-graded object $F_0(C_0)$ by $C_0$.

For an $m$-graded $E_k$-ring $A$, the counit of this adjunction provides a map 
$A_0 \to A$ of $m$-graded $E_k$-rings.  If $A$ is  concentrated in weight $0$, the counit of $F_0\dashv G_0$ provides  an equivalence 
\begin{equation}\label{eq weight zero counit map is an equivalence}
    F_0G_0(A) \simeq A.
\end{equation} 

 \section{Complex $K$-theory spectrum as a $p-1$-graded $E_\infty$-ring}\label{sec einfty grading on ku}
 Here, we use the results of \cite{carmeli2021chromatic} to obtain a $p-1$-graded $E_\infty$-ring structure on $ku_p$. Furthermore, we show that the resulting $p-1$-graded $E_1$-ring structure on $ku_p$ agrees with that provided by the root adjunction methods of \cite{ausoni2022adjroot}. 
 
 Similarly, we construct a $2$-periodic $p-1$-graded Morava $K$-theory spectrum of height $1$, i.e.\ $ku_p/p$, as a $p-1$-graded $E_\infty$ $ku_p$-algebra. 
 \subsection{Complex $K$-theory spectrum}

 Recall that $L_p \to KU_p$ is a Galois extension with Galois group $\Delta := \z/(p-1)$ in the sense of \cite[Section 5.5.4]{rognes2008galois}.  Taking connective covers, one obtains that $ku_p$ is a $\Delta$-equivariant  $E_\infty$ $\ell_p$-algebra and there is an equivalence:
 \[\ell_p \simeq ku_p^{h\Delta}.\]

  We consider $\Delta$ as the cyclic subgroup $\z/(p-1)$ of $\zp^{\times}$. Let $\delta$ denote a generator of $\Delta$ and $\alpha$ denote the corresponding element in $\zp$. On $\pis ku_p \cong \zp[u_2]$, we have $\pis (\delta) (u_2^i) = \alpha^i u_2^i$. Note that since $\lv \Delta \rv = p-1$ and since $ku_p$ is $p$-local, the homotopy fixed points above can be computed by taking fixed points at the level of homotopy groups. 
  
  The $p-1$-graded $E_\infty$ $\ell_p$-algebra structure on $ku_p$ is a consequence of the Fourier transform developed in \cite[Section 3]{carmeli2021chromatic}. Due to \cite[Corollary 3.9]{carmeli2021chromatic}, $\ell_p$ admits a primitive $p-1$-root of unity in the sense of \cite[Definition 3.3]{carmeli2021chromatic} which we can choose to be $\alpha \in \zp \cong \pi_0 \ell_p$ above. Let $\Delta^*$ denote the Pontryagin dual 
  \[\Delta^* := \hom(\Delta, \z/(p-1))\] 
  of $\Delta$ for which we have $\Delta^*\cong \z/(p-1)$.  In this situation, \cite[Proposition 3.13]{carmeli2021chromatic} provides a symmetric monoidal functor: 
 \begin{equation}\label{eq fourier transform}
     \mathfrak{F} \co \lmod_{\ell_p}^{B\Delta} \to \fun(\z/(p-1), \lmod_{\ell_p})
 \end{equation}
  from the $\infty$-category of $\Delta$-equivariant $\ell_p$-modules to the $\infty$-category of $p-1$-graded $\ell_p$-modules. Indeed, this functor is an equivalence of $\infty$-categories. For the $E_\infty$-algebra $\kup$ in $\lmod_{\ell_p}^{B\Delta}$, we will show that $\mathfrak{F}(ku_p)$ provides the desired $p-1$-graded $E_\infty$ $\ell_p$-algebra structure on  $ku_p$.

First, we describe $\mathfrak{F}(ku_p)$ as a $p-1$-graded $\ell_p$-module. Indeed, $\mathfrak{F}$ provides the underlying eigenspectrum decomposition as described in \cite[Remark 3.14]{carmeli2021chromatic}. Namely, by \cite[Definition 3.12]{carmeli2021chromatic} we have 
\[\mathfrak{F}(ku_p)_i \simeq (\ell_p(-\varphi_{i}) \wdg_{\ell_p} \kup)^{h\Delta}\]
where $\ell_p(-\varphi_i)$ is given by $\ell_p$ as an $\ell_p$-module but $\delta$ acts through multiplication by $\alpha^{-i}$ on $\pis (\ell_p(-\varphi_i))$ \cite[Definition 3.10]{carmeli2021chromatic}; here, $\varphi_i$ is the map $\z/(p-1) \to \z/(p-1)$ that multiplies by $i$. Again, homotopy fixed points may be computed by taking fixed points at the level of homotopy groups and one observes that 
\[\pis\big (\mathfrak{F}(ku_p)_i\big)\cong \pis\big( (\ell_p(-\varphi_i) \wdg_{\ell_p} \kup)^{h\Delta}\big )\]
is precisely given by the eigenspace corresponding to $\alpha^i$ in $\pis ku_p$. This eigenspace is  $\pis (\Sigma^{2i}\ell_p) \subseteq \pis ku_p$. In particular, $\pis\big (\mathfrak{F}(ku_p)_i\big)$ is free of rank $1$ as a $\pis \ell_p$-module and therefore, we obtain equivalences of $\ell_p$-modules  
\[\mathfrak{F}(ku_p)_i \simeq \Sigma^{2i}\ell_p.\]

Therefore, the underlying $\ell_p$-module of the $p-1$-graded $\ell_p$-module $\mathfrak{F}(ku_p)$ is given by 
\begin{equation}\label{eq underlying spectrum of graded ku is ku}
    D(\mathfrak{F}(ku_p)) \simeq \bigvee_{0\leq i <p-1}\Sigma^{2i} \ell_p \simeq ku_p
\end{equation}
as desired. The following proposition identifies the underlying $E_\infty$ $\ell_p$-algebra of $\mathfrak{F}(ku_p)$ with $ku_p$. 
\begin{prop}\label{prop specific graded einfty on kup}
There is an equivalence of $E_\infty$ $\ell_p$-algebras
\[D(\mathfrak{F}(ku_p)) \simeq ku_p.\]
\end{prop}
\begin{proof}
Using the strong monoidality of $\mathfrak{F}$, one observes that there is an isomorphism $\pis  \big(D(\mathfrak{F}(ku_p))\big) \cong \pis ku_p$ of $\pis \ell_p$-algebras. Inverting $v_1 \in \pis \ell_p$, we obtain the following commuting diagram of $E_\infty$-rings.
\begin{equation}\label{diag inverting vone in lp and kup}
\begin{tikzcd}
\ell_p \ar[d]\ar[r] &D(\mathfrak{F}(ku_p)) \ar[d]\\
L_p \ar[r] & D(\mathfrak{F}(ku_p))[v_1^{-1}]
\end{tikzcd}
\end{equation}

There is an isomorphism of $\pis L_p$-algebras 
\begin{equation*}\label{eq another isomorphism at homotopy rings}
    \pis \big(D(\mathfrak{F}(ku_p))[v_1^{-1}] \big) \cong \pis KU_p.
\end{equation*}
Since $\pis L_p \to \pis KU_p$ is an \'etale map of Dirac rings \cite[Example 4.32]{hesselholt2022dirac}, we deduce by \cite[Theorem 1.10]{hesselholt2022dirac} that the isomorphism above lifts to an equivalence of $E_\infty$ $L_p$-algebras
\[D(\mathfrak{F}(ku_p))[v_1^{-1}] \simeq KU_p.\]
Alternatively, this also follows by  \cite[Proposition 2.2.3]{baker2007realizibilityofalgebraicgalois}.

Since the right hand vertical arrow in Diagram \eqref{diag inverting vone in lp and kup} is a connective cover, the universal property of connective covers (in $E_\infty$ $\ell_p$-algebras) provide an equivalence 
\[D(\mathfrak{F}(ku_p)) \simeq ku_p\]
of $E_\infty$ $\ell_p$-algebras. 
\end{proof}
  
 \begin{theo}\label{theo kup is a p-1 graded einfty lp algebra}
  The  $E_\infty$ $\ell_p$-algebra  $ku_p$ admits the structure of a $p-1$-graded $E_\infty$ $\ell_p$-algebra such that 
  \[(ku_p)_i \simeq \Sigma^{2i}\ell_p.\]
 \end{theo}
  
  \begin{rema}
   From this point, when we mention $ku_p$ as a  $p-1$-graded $E_\infty$ $\ell_p$-algebra, we mean $\mathfrak{F}(ku_p)$. 
  \end{rema}
  
  \begin{rema}
  We would like to thank Tommy Lundemo for pointing out that it should also be possible to construct a $p-1$-graded $E_\infty$-algebra structure on $ku_{(p)}$ using  \cite[Proposition 4.15]{sagave2014logarithmic} which states that $ku_{(p)}$ can be obtained from $\ell$ via base change through the polynomial like $E_\infty$-algebras of  \cite[Construction 4.2]{sagave2014logarithmic}.
  \end{rema}
 \subsection{Adjoining roots to ring spectra}\label{subsec adj roots to ring spectra}
 Here, we summarize the root adjunction method developed in \cite[Construction 4.6]{ausoni2022adjroot}. 
 
 Let $k>0$ be even and let $\sphzk$ denote the free $E_1$-ring spectrum on $\sph^k$ (this is denoted by $\sphsk$ in \cite{ausoni2022adjroot}). Taking $z_k$ to be of weight $1$, $\sphzk$ admits the structure of a graded $E_2$-algebra \cite[Construction 3.3]{ausoni2022adjroot}. 
 By left Kan extending $\sphzk$ through $\z\to \z/m$, one obtains an $m$-graded $E_2$-ring that we also call $\sphzk$ with $z_k$ in weight $1$.  

 Let $A$ be an $E_1$ $\sph[z_{mk}]$-algebra where $z_{mk}$ acts through $a \in \pi_{mk}A$. Using $F_0$, we obtain a map of $m$-graded $E_2$-rings concentrated in weight $0$: \[\sphzmk \to A.\] 
 
 Furthermore, \cite[Proposition 3.9]{ausoni2022adjroot} provides a map $\sph[z_{mk}] \to \sph[z_k]$ of $m$-graded $E_2$-rings carrying the weight $0$ class $z_{mk}$ to $z_k^m$ in homotopy. Finally, the $m$-graded $E_1$-ring $A(\sqrt[m]{a})$ is defined via the following relative smash product in $m$-graded spectra:
 \begin{equation}\label{eq defining root adjunction}
 A(\sqrt[m]{a}) :=A \wdg_{\sph[z_{mk}]} \sph[z_k].
 \end{equation}

This comes equipped with a map $A\to  \artma$ of  $m$-graded $E_1$-rings given by the counit of the adjunction $F_0 \dashv G_0$.

It follows by the K\"unneth spectral sequence that at the level of homotopy rings, one obtains precisely the desired root adjunction:
\begin{equation}\label{eq homotopy ring after root adjunction}
    \pis\big( \artma\big) \cong \pis( A) [y]/(y^m-a).
\end{equation}
The class $y$ above comes from $z_k \in \pis \sph[z_k]$ and therefore it is of weight $1$. Furthermore, $\pis A \subseteq \pis \big(\artma\big)$ is the subring of weight $0$ elements.


\begin{lemm}\label{lemm m graded e2 rings have even cell decomposition}
Let $k\geq0$ be even. The $m$-graded $E_2$-ring obtained from the graded $E_2$-ring $\sph[z_k]$ by left Kan extending through $\z \to \z/m$ admits an even cell decomposition. 
\end{lemm}
\begin{proof}
This follows as in the proof of \cite[Lemma 3.6]{ausoni2022adjroot}. The $E_2$-ring $\sphzk$ admits an even cell decomposition \cite[Proposition 3.4]{ausoni2022adjroot}, i.e.\ it is given by a filtered colimit of graded $E_2$-rings starting with the  free graded $E_2$-algebra on $\sph^{k}$ and the later stages  given by attaching an even $E_2$-cell to the former. Note that left Kan extension through $\z \to \z/m$  is left adjoint and symmetric monoidal. Therefore, it preserves free algebras, even cell attachments and filtered colimits. This  provides the $m$-graded $E_2$-ring $\sph[z_k]$ with an even cell decomposition.  
\end{proof}

 \subsection{Complex $K$-theory spectrum via root adjunction}\label{subsec ku via root adjunction}
 
Here, we show that the $p-1$-graded $E_1$ $\ell_p$-algebra structure on $ku_p$ provided by Proposition \ref{prop specific graded einfty on kup} agrees with that obtained by  adjoining a root to $v_1 \in \pis \ell_p$. 

Let $\sph[z_2]$ be the $p-1$-graded $E_2$-algebra with $z_2$ in weight $1$ as mentioned earlier. Since $\pis \ell_p$ is concentrated in even degrees, Lemma \ref{lemm m graded e2 rings have even cell decomposition}  provides a $p-1$-graded $E_2$-ring map 
\begin{equation}\label{eq free spherical to kup}
    \sphztwo \to ku_p
\end{equation}
 that carries $z_2$ to $u_2$ in homotopy. 
The aforementioned $p-1$-graded $E_2$-map $\sphzpmone \to \sphztwo$ carrying the weight $0$ class $z_{2(p-1)}$ to $z_2^{p-1}$ induces  the second equivalence below.
\[\sphzpmone \simeq F_0G_0(\sphzpmone) \simeq F_0G_0(\sphztwo) \]
The first equivalence follows by \eqref{eq weight zero counit map is an equivalence} and these are equivalences of $p-1$-graded $E_2$-rings.

 Similarly, the $p-1$-graded $E_\infty$-ring map $\ell_p \to ku_p$ (with $\ell_p$ concentrated in weight $0$), provides an equivalence $\ell_p \simeq F_0G_0(ku_p)$ of $p-1$-graded $E_\infty$-rings concentrated in weight $0$. We obtain the following commuting diagram of $p-1$-graded $E_2$-rings by applying the natural transformation $F_0G_0 \to id$ to \eqref{eq free spherical to kup} and using the last two equivalences we mentioned above. 
\begin{equation}\label{diag maps from the weight 0 section of sph z2 to kup}
\begin{tikzcd}
\sphzpmone \ar[r] \ar[d] & \ell_p \ar[d]\\
\sphztwo \ar[r] &ku_p
\end{tikzcd}
\end{equation}
In particular, the map $\ell_p \to ku_p$ is a map of $p-1$-graded $\sphzpmone$-algebras. The extension/restriction of scalars adjunction induced by the left hand vertical map provides a map 
\[\ell_p \wdg_{\sphzpmone} \sphztwo \xrightarrow{\simeq} ku_p \]
of $p-1$-graded $E_1$ $\sphztwo$-algebras. Note that the left hand side above is a form of $\ell_p(\sqrt[p-1]{v_1})$ as in \eqref{eq defining root adjunction}. Considering  \eqref{eq homotopy ring after root adjunction}, one observes that the map above is an equivalence as desired. This proves the following. 
\begin{prop}\label{prop adj root to lp is kup}
Let $ku_p$ denote a $p-1$-graded $E_1$-ring provided by Theorem \ref{theo kup is a p-1 graded einfty lp algebra}. Then there is an equivalence of $p-1$-graded $E_1$-rings
\[ku_p \simeq \ell_p(\sqrt[p-1]{v_1})\]
for the form of $\ell_p(\sqrt[p-1]{v_1})$ constructed above.
\end{prop}

\subsection{Two periodic Morava K-theory as a $p-1$-graded $ku_p$-algebra} \label{subsec two periodic morava kthry via root adjunction}
Using the even cell decomposition of $\sph[z_0]$, we obtain a map $\sph[z_0] \to \ell_p$ of $E_2$-rings that carries $z_0$ to $p$ in homotopy. Through this, we define the connective first Morava $K$-theory $k(1)\simeq \ell/p$ as an $E_1$ $\ell_p$-algebra as follows:
\[\ell/p:= \sph \wdg_{\sph[z_0]} \ell_p.\]
Here, we make use of the $E_\infty$ map $\sph[z_0] \to \sph$ sending $z_0$ to $0$; this is the weight $0$-Postnikov truncation of $\sph[z_0]$ (\cite[Lemma B.0.6]{hahn2020redshift}); alternatively, this is the map  $\Sigma^\infty_+ \mathbb{N} \to \Sigma^\infty_+ *$. Using $F_0$, we consider $\ell/p$ as a $p-1$-graded $\ell_p$-algebra concentrated in weight $0$.

 We define the connective two periodic first Morava $K$-theory  $ku/p$ as a $p-1$-graded $ku_p$-algebra as follows. 
\[ku/p := \ell/p \wdg_{\ell_p} ku_p\]
\begin{prop}\label{prop two periodic morava kthry is given by root adjunction}
There is an equivalence of $p-1$-graded $E_1$-rings
\[ku/p \simeq \ell/p(\sqrt[p-1]{v_n})\]
for some form of $\ell/p(\sqrt[p-1]{v_n})$.
\end{prop}
\begin{proof}
There is a map of $p-1$-graded $\ell_p$-algebras
\[\ell/p \to ku/p:= \ell/p \wdg_{\ell_p} ku_p.\]
The target carries a $p-1$-graded $ku_p$-algebra structure compatible with its $p-1$-graded $\ell_p$-algebra structure. Forgetting through Diagram \eqref{diag maps from the weight 0 section of sph z2 to kup}, this is a map of $p-1$-graded $\sphzpmone$-algebras where the target carries a compatible $p-1$-graded $\sphztwo$-algebra structure. Extending scalars, we obtain a map of $p-1$-graded $E_1$-rings:
\[ \ell/p \wdg_{\sphzpmone} \sphztwo\xrightarrow{\simeq}ku/p,\]
which can easily shown to be an equivalence. 

\end{proof}

\section{Graded THH, $\ntc$ and TC}\label{sec graded THH, TC, neg TC}

Let $X$ be a $p$-local $m$-graded $E_k$-ring for $m>0$. In \cite[Appendix A]{antieau2020beilinson}, the authors prove that in this situation, $\thh(X)$ is an $m$-graded $E_{k-1}$-algebra in $\spcircle$. In particular, $\thh(X)$ admits an $S^1$-equivariant splitting into a coproduct of $m$-cofactors. Since the homotopy fixed points functor and the Tate construction commute with finite coproducts, this  splits  $\thh(X)^{hS^1}$ and $(\thh(X)^{tC_p})^{hS^1}$ as well. However, these splittings may not carry over to $\tc(X)$ in general since the canonical map is given by maps $\thh(X)_i^{hS^1} \to (\thh(X)_i^{tC_p})^{hS^1}$ that preserve the $m$-grading whereas the Frobenius map is given by maps
\[\thh(X)_{i} \to \thh(X)^{tC_p}_{pi}.\]
In particular, the fiber sequence defining $\tc(X)$
\[\tc(X) \to \bigvee_{i \in \z/m} \thh(X)_i^{hS^1}  \xrightarrow{\varphi_p^{hS^1} - can} \bigvee_{i \in \z/m} (\thh(X)_i^{tC_p})^{hS^1}\]
 may not split. On the other hand, if $m \mid p-1$, then $p=1$ in $\z/m$ and the Frobenius map also respects the $m$-grading. This results in a splitting of the fiber sequence defining $\tc(X)$ and hence a splitting of $\tc(X)$ into $m$-factors. Since we eventually work with $p-1$-graded spectra, this applies to our examples. In this section, we make this precise and deduce that $\tc(ku_p)$ is a $p-1$-graded $E_\infty$-ring and that  $\tc(ku/p)$ is a $p-1$-graded $\tc(ku_p)$-module. 

 For a given spectrum $F$, we let  $F^{\textup{triv}}$ denote the cyclotomic spectrum with trivial $S^1$-action and the Frobenius map given by the composite $F \to F^{hC_p}\to F^{tC_p}$ where the first  map comes from the fact that $F$ has the trivial action and the second map is the canonical map. 

\begin{defi}\label{defi of tc and trivial cyclotomic spectrum functor}
Since $\cycsp$ is a stable and presentably symmetric monoidal $\infty$-category, it follows by Shipley's theorem that there is a unique cocontinuous symmetric monoidal functor 
\[(-)^{\textup{triv}}\co \sp \to \cycsp\] given by the trivial cyclotomic structure described above. 

The right adjoint to $(-)^{triv}$ is the lax symmetric monoidal functor 
\[\tc \co \cycsp \to \sp\]
given by 
\[\tc(-) \simeq \map_{\cycsp}(\sph^{\textup{triv}},-).\]
\end{defi}

For the rest of this section, assume that $m$ is a positive integer such that $m \mid p-1$. Using the results of \cite[Appendix A]{antieau2020beilinson}  we prove in Appendix \ref{appendix graded tc} below that there is a symmetric monoidal functor 
\[\alg_{E_1}(\fun(\z/m,\sp)) \xrightarrow{\textup{THH}} \fun(\z/m,\cycsp).\]
Furthermore, it follows by \cite[Corollary 3.7]{nikolaus2016stablemultyoneda} that the levelwise application of $\tc$ provides a lax symmetric monoidal functor: 
\[\tc \co \fun(\z/m,\cycsp) \to \fun(\z/m,\sp),\]
that we also call $\tc$. In Appendix \ref{appendix graded tc}, we prove that the following diagram of lax symmetric monoidal functors commutes. 
\begin{equation}\label{diag for graded thh and graded tc}
\begin{tikzcd}
    \alg_{E_1}(\fun(\z/m,\sp)) \ar[d] \ar[r,"\thh"] & \fun(\z/m,\cycsp) \ar[r,"\tc"] \ar[d] & \fun(\z/m,\sp)\ar[d]\\
    \alg_{E_1}(\sp) \ar[r,"\thh"] & \cycsp \ar[r,"\tc"] & \sp
    \end{tikzcd}
\end{equation}
The vertical maps above are given by left Kan extension along $\z/m\to 0$, i.e.\ they provide the underlying objects. 

\begin{rema}
The composite $\tc \circ \thh$ at the bottom row above may not in general give the correct result since we only consider one prime in $\cycsp$. However, this is not an issue since we only work with $p$-complete objects in our applications.
\end{rema}

\begin{cons}\label{cons tc of kup and ku mod p are graded}
Since $ku_p$ is a $p-1$-graded  $E_\infty$-ring, we obtain that $\thh(ku_p)$ is a $p-1$-graded $E_\infty$-algebra in cyclotomic spectra and that $\tc(ku_p)$ is a $p-1$-graded $E_\infty$-ring. 

Furthermore, in Section \ref{subsec two periodic morava kthry via root adjunction}, we defined $ku/p$ as a $p-1$-graded $ku_p$-algebra. In particular, this implies that $ku/p$ is a right module over $ku_p$ in the $\infty$-category of $p-1$-graded $E_1$-rings, see \cite[Construction 4.11]{ausoni2022adjroot}. Therefore,  $\thh(ku/p)$ is a  right $\thh(ku_p)$-module in the $\infty$-category of $p-1$-graded cyclotomic spectra and that  $\tc(ku/p)$ is a right $\tc(ku_p)$-module in the $\infty$-category of $p-1$-graded spectra. 
\end{cons}

\begin{rema}
Furthermore, the levelwise application of the symmetric monoidal functor $\cycsp \to \sp^{BS^1}$ that forgets the Frobenius map shows that $\thh(X)$ is an $m$-graded $E_{k-1}$-algebra in $\sp^{BS^1}$ whenever $X$ is an $m$-graded $E_k$-ring. In particular, $\thh(X)^{hS^1}$ and $(\thh(X)^{tC_p})^{hS^1}$ also admit the structures of $m$-graded $E_{k-1}$-algebras.
\end{rema}

\subsection{Weight zero splitting of THH for root adjunctions.}

In Section \ref{subsec ku via root adjunction}, we show that the $p-1$-graded $E_1$-ring structure on $ku_p$ agrees with that given by the root adjunction method of \cite{ausoni2022adjroot}. The reason we do this is so that we can make use of Theorem 4.17 of \cite{ausoni2022adjroot} which states that for a $p$-local $A$, $\thh(A) \to \thh(A(\sqrt[m]{a}))_0$ is an equivalence whenever $p \nmid m$. Furthermore, this equivalence carries over to topological cyclic homology due to \cite[Theorem 5.5]{ausoni2022adjroot}. We obtain the following.
\begin{prop}\label{prop weight zero inclusion for thh ku}
The canonical maps
\[\thh(\ell_p) \xrightarrow{\simeq} \thh(ku_p)_0 \textup{\ \ and\ } \tc(\ell_p) \xrightarrow{\simeq} \tc(ku_p)_0\]
are equivalences. 
\end{prop}

\begin{prop}\label{prop weight zero inclusion of thh two periodic morava kthry}
The canonical maps
\[\thh(\ell/p) \xrightarrow{\simeq} \thh(ku/p)_0  \textup{\ \ and \ } \tc(\ell/p) \xrightarrow{\simeq} \tc(ku/p)_0\]
are equivalences.
\end{prop}

\section{Logarithmic THH of the complex $K$-theory spectrum}\label{sec log thh with an equivariant splitting}

Here, we use the $p-1$-grading on $\thh(ku_p)$ to obtain a splitting of the logarithmic THH of $ku_{(p)}$ as an $S^1$-equivariant spectrum. We identify the resulting splitting at the level of  $V(1)$-homotopy  by using the logarithmic THH computations of Rognes, Sagave and Schlichtkrull in  \cite{rognes2018logthhofku}. For the rest of this section, let $p>3$.

\begin{rema}
For the rest, we consider $V(1) \to T(2)$ as a map of commutative monoids in the homotopy category of $p-1$-graded cyclotomic spectra with the trivial cyclotomic structure concentrated in weight $0$ (using   $F_0 \circ (-)^{triv}$).  
\end{rema}
\begin{rema}
In the following, we move freely between $\thh(ku_{(p)})$ and $\thh(ku_{p})$ since ultimately, we are interested in the $V(1)$-homotopy of these objects for which we have an equivalence $V(1) \wdg \thh(ku_{(p)}) \simeq V(1) \wdg \thh(ku_p)$. Similarly, we move freely between $V(1) \wdg \thh(\ell)$ and $V(1) \wdg \thh(\ell_p)$.

\end{rema}

Let $\thh(ku_{(p)} \mid u_2)$ denote the logarithmic THH of $ku_{(p)}$ with respect to the Bott class $u_2 \in \pis(ku_{(p)})\cong \zpl[u_2]$ in the sense of \cite{rognes2018logthhofku}. In \cite{rognes2018logthhofku}, this is denoted by $\thh(ku_{(p)}, D(u))$. This is an $S^1$-equivariant $E_\infty$-algebra and there is a cofiber sequence of $S^1$-equivariant spectra: 
\begin{equation}
    \thh(ku_{(p)}) \to \thh(ku_{(p)}\mid u_2) \to \Sigma \thh(\zpl),
\end{equation}
where the first map is a map of $E_\infty$-algebras in $S^1$-equivariant spectra, see the discussion after \cite[Definition 4.6]{rognes2015localization}. 

Here, our goal is to prove the following proposition where $\thh(\ell\mid v_1)$ denotes the logarithmic THH of $\ell$ with respect to the class $v_1 \in \pis \ell$ as defined in \cite{rognes2018logthhofku} where it is denoted by $\thh(\ell, D(v))$. 
\begin{prop}\label{prop weight grading on log thh ku}
There is an equivalence of $S^1$-equivariant spectra:
\[\vo \wdg \thh(ku_{(p)}\mid u_2)\simeq \vo \wdg \thh(\ell\mid v_1) \vee \big(\bigvee_{i \in \z/(p-1) \mid i \neq 0} \vo \wdg \thh(ku_p)_i\big),\]
given by the coproduct of the map $\vo \wdg \thh(\ell\mid v_1) \to \vo \wdg \thh(ku_{(p)} \mid u_2)$ with the composite:
\[\bigvee_{i\in \z/(p-1)\mid i \neq0 }\vo  \wdg \thh(\kup)_i \to \vo \wdg \thh(ku_p) \to \vo \wdg \thh(\kupl \mid u_2),\]
where the first map is given by the inclusion of the  given summands of the $p-1$-graded spectrum $\thh(ku_p)$ and the second one is the canonical one.

\end{prop}

\begin{rema}
 Since $ku_p \simeq \ell_p (\sqrt[p-1]{v_1})$, this is an immediate consequence of the results of \cite[Section 6]{ausoni2022adjroot} if we assume that the definition of logarithmic THH in \cite{ausoni2022adjroot} agrees with that used in \cite{rognes2018logthhofku}. This compatibility result is not available at the moment, and therefore, we will not assume it. On the other hand, Devalapurkar and Moulinos prove this compatibility result in their upcoming work. 
\end{rema}

\begin{proof}
Due to \cite[Theorem 4.4]{rognes2018logthhofku},  there is a map of homotopy cofiber sequences of $S^1$-equivariant spectra: 
\begin{equation*}
    \begin{tikzcd}
        V(1) \wdg \thh(\zpl) \ar[r]\ar[d] &  V(1) \wdg \thh(\ell) \ar[r]\ar[d]& V(1) \wdg \thh(\ell \mid v_1)\ar[d] \\
          V(1) \wdg \thh(\zpl) \ar[r] &  V(1) \wdg \thh(ku_p) \ar[r] & V(1) \wdg \thh(ku_{(p)} \mid u_2)
    \end{tikzcd}
\end{equation*}
as mentioned in \cite[Equation (8.1)]{rognes2018logthhofku}. Here, the left hand vertical map is an equivalence. Therefore, the bottom left horizontal map factors as 
\[V(1) \wdg \thh(\zpl) \to  \vo \wdg \thh(\ell) \to \vo \wdg \thh(ku_p) \simeq \bigvee_{i \in \z/(p-1)} V(1) \wdg \thh(ku_p)_i.\]
The second map above is  the inclusion of the weight $0$ summand due to Proposition \ref{prop weight zero inclusion for thh ku}. In particular, the cofiber sequence given by the bottom row splits through the splitting of $\thh(ku_p)$. Namely, this cofiber sequence is given by a coproduct of the cofiber sequence 
given by the top row and the cofiber sequence 
\[* \to \bigvee_{i \in  \z/(p-1) \mid i \neq 0} V(1) \wdg \thh(ku_p)_i\xrightarrow{\simeq} \bigvee_{i \in  \z/(p-1) \mid i \neq 0} V(1) \wdg \thh(ku_p)_i.\]
This identifies the cofiber, i.e.\ $\vo \wdg \thh(\kupl \mid u_2)$ as stated in the proposition. 

\end{proof}

We will identify the homotopy groups of the summands of $V(1) \wdg  \thh(\kupl\mid u_2)$ given by the splitting above. For this, we start by recalling the computations of $V(1)_* \thh(\ell\mid v_1)$ and $V(1)_* \thh(\kupl\mid u_2)$ from \cite{rognes2018logthhofku}. For what follows, $E(x,y)$, $P(x)$  and $P_k(x)$ denote the exterior algebra over $\fp$ in two variables,  the  polynomial algebra $\fp[x]$ and the truncated polynomial algebra $\fp[x]/x^k$ respectively. 

\begin{theo}\label{theo log thh of ku and ell}\cite[Theorems 7.3 and 8.1]{rognes2018logthhofku}
There are ring isomorphisms:
\begin{equation*}
    \begin{split}
        \vos \thh(\ell \mid v_1) \cong&  E(\lambda_1,d\log v_1) \otimes P(\kappa_1) \\
        \vos \thh(\kupl \mid u_2) \cong& P_{p-1}(u_2) \otimes E(\lambda_1,d\log u_2) \otimes P(\kappa_1)
    \end{split}
\end{equation*}
where $\lv \lambda_1\rv = 2p-1$, $\lv \kappa_1 \rv = 2p$, $\lv d\log v_1\rv = \lv d\log u_2\rv = 1$ and $\lv u_2 \rv = 2$. Furthermore, the map \[\vos \thh(\ell \mid v_1) \to \vos \thh(\kupl \mid u_2)\]
is given by the ring map that carries $d \log v_1$ to $- d \log u_2$, $\lambda_1$ to $\lambda_1$ and $\kappa_1$ to $\kappa_1$. 
\end{theo}

 Recall that there is an action of the group  $\Delta := \z/(p-1)$ on  $ku_p$ through Adams operations. Let $\delta \in \Delta$ be a chosen generator and we choose a $\beta \in \fp^\times$ such that  $\pis (\sph/p \wdg \delta) (u_2) = \beta u_2$; here,
\[\pis (\sph/p \wdg \delta) \co \pis (\sph/p \wdg \kup) \to \pis (\sph/p \wdg \kup) \cong \fp[u_2].\]

 A given $x \in V(1)_*\thh(\kup)$ is said to have $\delta$-weight $i\in \z/(p-1)$ if the automorphism of  $V(1)_*\thh(\kup)$ induced by $\delta$ carries $x$ to $\beta^ix$  \cite[Definition 8.2]{ausoni2005thhofku}; the $\delta$-weights of the generators of $\vos \thh(\kup)$ are given in \cite[Proposition 10.1]{ausoni2005thhofku}. One defines $\delta$-weight in a similar way for $\vos \kth(\kup)$, $\vos\tc(\kup)$ etc.

It follows by \cite[Proposition 8.2]{ausoni2022adjroot} that $V(1)_*\thh(ku_p)_i$ is precisely given by the classes of $\delta$-weight $i$ in $\vos \thh(\kup)$. In other words, $\delta$-weight and our weight gradings agree for $\vos \thh(ku_p)$. 
\begin{prop}\label{prop subgroups of the weight grading on log thh ku}
For $0<i < p-1$, the image of the inclusion 
\[\psi_i \co \vos \thh(ku_p)_i \to \vos \thh(\kupl \mid u_2)\]
is given by:
\[\{u_2^{i}\} \otimes E(\lambda_1,d\log u_2) \otimes P(\kappa_1).\]
Here, the maps $\psi_i$ are given by Proposition \ref{prop weight grading on log thh ku}.
\end{prop}
\begin{proof}
It follows from Theorem  \ref{theo log thh of ku and ell}  that the image of the inclusion 
\[\psi_0 \co \vos \thh(\ell
\mid v_1) \to \vos \thh(\kupl\mid u_2)\]
is given by 
\[V_0 := \{1\} \otimes E(\lambda_1,d\log u_2) \otimes P(\kappa_1);\]
 we say $V_0 = \textup{im}\  \psi_0$.
Also, let 
\[V_i = \{u_2^{i}\} \otimes E(\lambda_1,d\log u_2) \otimes P(\kappa_1).\]
 It follows by inspection on  \cite[Theorem 8.5]{rognes2018logthhofku} that  every $\fp$-module generator of $V_i$ given above gets hit by an element of $\delta$-weight $i$ under the map 
 \[\vos\thh(ku_{(p)})  \to \vos \thh(\kupl \mid u_2)\footnote{For this, note that the classes $a_i$, $b_i$ and $u_2$ are of $\delta$-weight $1$ and the classes $\lambda_1$ and $\mu_2$ are of $\delta$-weight $0$ in $\vos \thh(ku_{p})$.}.\]
 Since $\delta$-weight $i$ elements of $\vos \thh(ku_p)$ correspond to the $\fp$-submodule 
 \[\vos \thh(\kup)_i \subseteq \vos\thh(\kup),\]
 we deduce that $V_i \subseteq \textup{im}\  \psi_i$ for every $i$. Since 
 \[\bigoplus_{i \in \z/p-1} V_i\cong \vos \thh(\kupl \mid u_2) \cong \bigoplus_{i \in \z/p-1} \textup{im}\  \psi_i,\]
 and since all the vector spaces  involved are finite dimensional at each homotopy degree, we deduce that $V_i = \textup{im}\  \psi_i$ as desired. Note that the second isomorphism above follows by Proposition \ref{prop weight grading on log thh ku}.

\end{proof}
\section{Topological cyclic homology of complex $K$-theory}\label{sec topological cyclic homology of kup and ku mod p}

Let $p>3$ for the rest of this section. Here, we compute  $T(2)_*\kth(ku_p)$ and $T(2)_*\kth(ku/p)$.

\begin{rema}
Since $V(1)$ is a finite spectrum, $V(1) \wdg -$ commutes with all constructions involving colimits and limits. For instance, it commutes with homotopy fixed points and one has $\tc(\vo \wdg E) \simeq \vo \wdg \tc(E)$ for every cyclotomic spectrum $E$. 
\end{rema}

\subsection{Higher Bott element}
In  \cite[Section 3]{ausoni2010kthryofcomplexkthry}, Ausoni constructs a non-trivial class $b \in \vo_{2p+2} \kth(\kup)$ of $\delta$-weight $1$, that he calls the higher Bott element,  by considering the units of $\kup$. Here is a quick sketch of the construction of $b$. The infinite loop space $SL_1(ku)$ is often denoted by $BU_{\otimes}$. This notation is justified by the fact that the product represents the tensor products of virtual line bundles. The Postnikov truncation $BU_{\otimes} \to K(\z,2)$ admits a section \cite[\RomanNumeralCaps{5}.3.1]{may1977einfspacesandeinfringspectra}
\[j \co K(\z,2) \simeq BU(1) \to BU_{\otimes}\]
which is also a map of infinite loop spaces. The map $j$ represents taking the  corresponding virtual line bundle of a line bundle. This provides the following composite map 
\[\Sigma_+^\infty K(\z,3) \xrightarrow{\Sigma_+^{\infty} Bj} \Sigma_+^\infty BBU_{\otimes}\simeq \Sigma_+^\infty BSL_1(ku) \to \Sigma_+^\infty BGL_1(ku) \to \kth(ku)\]
of $E_\infty$-rings. The class $b$ is the image of a class $b'$ in  $ V(1)_{2p+2} \Sigma^\infty_+ K(\z,3)$ under the $V(1)$-homotopy of the composite map above. To construct $b'$, Ausoni begins with  obtaining low degree information on $V(1)_* K(\z,2)$ using Cartan's computation of $H_*(K(\z,2),\fp)$. After this, $b'$ is constructed as a filtration degree $2$ class in the $V(1)$-homotopy of the bar filtration on the classifying space $K(\z,3) \simeq BK(\z,2)$. 

Let $b\in \vo_{2p+2} \tc(\kup)$ also denote the image of this class under the map $\vos \kth(\kup) \to \vos \tc(\kup)$; this is also a non-trivial class due to the following proposition. 
\begin{prop}\label{prop properties of the higher Bott element b}
The classes $b$ mentioned above satisfies the following properties. 
\begin{enumerate}
    \item The map $ \vos \tc(\kup)\to \vos \thh(\kup)$ carries $b$ to a $\delta$-weight $1$ class denoted as $b_1$ in \cite{ausoni2005thhofku}, see \cite[Lemma 4.4]{ausoni2010kthryofcomplexkthry}. Since $b_1$ is of $\delta$-weight $1$, we have \[b_1 \in \vos \thh(ku_p)_1.\]
    \item The map $\vos \thh(ku_p) \to \vos \thh(\kupl\mid u_2)$ carries $b_1$ to $u_2 \kappa_1$ \cite[Theorem 8.5]{rognes2018logthhofku}.
    \item In $\vos \kth(\kup)$, we have $b(b^{p-1} + v_2) = 0$ \cite[Proposition 2.7]{ausoni2010kthryofcomplexkthry}.
\end{enumerate}
\end{prop}

Note that we do not provide a proof for the proposition above since everything stated are immediate from the cited references. 

\begin{prop}\label{prop higher bott is homogenous of weight 1}
The higher Bott element $b\in \vos \tc(\kup)$ is a homogeneous element of weight $1$ in the $p-1$-grading. In other words, 
$b \in \vos \tc(ku_p)_1$. Similarly, the corresponding element $b \in \vos \thh(ku_p)^{hS^1}$ is also of weight $1$.

\end{prop}
\begin{proof}
Indeed, we show that all the elements in  $\vos\tc(\thh(ku_p)_i)$ are of $\delta$-weight $i$. Since $\Delta = \z/(p-1)$ is an abelian group, the map $ \delta \co ku_p \to ku_p$ induced by the chosen generator $\delta \in \Delta$ is a map of $E_\infty$-algebras in the $\infty$-category of $\Delta$-equivariant $\ell_p$-modules (not just a map of $E_\infty$ $\ell_p$-algebras). Therefore, using $\mathfrak{F}$ in \eqref{eq fourier transform}, $\delta \co ku_p \to ku_p$ can be considered as a  map of $p-1$-graded $E_\infty$ $\ell_p$-algebras. 

As a result, the induced map $\thh(\delta)$ is a map of $p-1$-graded cyclotomic objects. In particular, it preserves weight at the level of $\tc$, $\ntc$ and $\tp$. Recall that each \[x \in \vos\thh(ku_p)_i\]  is of $\delta$-weight $i$. Therefore, the map induced by $\delta$ at the level of the homotopy fixed point spectral sequence for  $\vos \thh(ku_p)_i^{hS^1}$ is given by multiplication by $\beta^{i} \in \fp^{\times}$. Since $\vos \thh(ku_p)_i$ is finite at each degree, this spectral sequence is strongly convergent \cite[Theorem 7.1]{boardman1998conditionalconv}. We deduce that every class in $\vos \thh(ku_p)_i^{hS^1}$   with defined $\delta$-weight have $\delta$-weight $i$. On the other hand, $\vos \thh(\delta)_i^{hS^1}$ is diagonalizable (since its $p-1$st power is identity), i.e.\  $\vos \thh(ku_p)_i^{hS^1}$  have a basis for which $\delta$-weight is defined for each basis element. Therefore, we deduce that all the classes in  $\vos \thh(ku_p)_i^{hS^1}$ are of $\delta$-weight $i$. The same argument shows that every class in  $\vos \thh(ku_p)_i^{tS^1}$ is of  $\delta$-weight $i$.

The fiber sequence defining  $\tc$ also shows that each class in $\vos\tc(\thh(ku_p)_i)$  either have $\delta$-weight $i$ or have undefined $\delta$-weight, but since this action is again diagonalizable, we deduce that every class in  $\vos\tc(\thh(ku_p)_i)$ is of $\delta$-weight $i$. Since $b$ is of $\delta$-weight $1$, the result follows.

\end{proof}

\subsection{Topological cyclic homology}
As mentioned earlier, we need to show that $b \in \ttws \tc(ku_p)$ is a unit. For this, we construct multiplication by $b$ as a self map of the cyclotomic spectrum $\vo \wdg \thh(ku_p)$ and show that it induces a self equivalence of $T(2) \wdg \tc(ku_p)$. We first show that $b$ provides a unit in $\ttws \thh(ku_p)^{hS^1}$ by comparing it with the corresponding multiplication in $T(2)_*\thh(ku_p \mid u_2)^{hS^1}$. 

\begin{cons}\label{cons the map of cyclotomic spectra given by multiplication by b}
We start with the map $\sph^{2p+2} \to \tc(\vo \wdg \thh(\kup))$ representing $b$. Using the adjunction $(-)^{\triv} \dashv \tc$ mentioned in Definition \ref{defi of tc and trivial cyclotomic spectrum functor}, one obtains a map of cyclotomic spectra 
\[b_1 \co \Sigma^{2p+2} \sph^{\triv} \to \vo \wdg \thh(\kup)\] 
representing the class $b_1$.  We define
\[m_b \co \Sigma^{2p+2} V(1) \wdg \thh(ku_p) \to V(1) \wdg \thh(ku_p)\]
as the following composite map of cyclotomic spectra.
\begin{equation*}
    \begin{split}
        m_b\co \Sigma^{2p+2} V(1) \wdg \thh(ku_p) \simeq V(1) \wdg \thh(ku_p) \wdg \Sigma^{2p+2} \sph^{\triv} &\xrightarrow{id \wdg b_1}\\ 
       \vo \wdg \thh(ku_p) \wdg \vo \wdg \thh(ku_p) \to& V(1) \wdg \thh(ku_p)
    \end{split}
\end{equation*}

Here, $id$ denotes the identity map of $\vo \wdg \thh(ku_p)$ and the second map above is given by the multiplication maps of $\thh(ku_p)$ and $V(1)$. 
\end{cons}

We construct a similar map for logarithmic THH of $\kupl$ which is compatible with the one constructed above. 
\begin{cons}\label{cons the map of log thh equivariant spectra multiplying by b}
The first map below is the  underlying $S^1$-equivariant map of the map $b_1$ in Construction \ref{cons the map of cyclotomic spectra given by multiplication by b}; the second map is the usual one.
\[u_2\kappa_1 \co \Sigma^{2p+2} \sph^{\triv} \xrightarrow{b_1}  \vo \wdg \thh(ku_p) \to \vo \wdg \thh(\kupl \mid u_2)\]
This composite is a map of $S^1$-equivariant spectra. Furthermore, it represents $u_2 \kappa_1$ in homotopy due to Proposition \ref{prop properties of the higher Bott element b}. As in Construction \ref{cons the map of cyclotomic spectra given by multiplication by b}, we define an $S^1$-equivariant map: 
\[m_{u_2 \kappa_1} \co \Sigma^{2p+2} \vo \wdg \thh(\kupl \mid u_2) \to \vo \wdg \thh(\kupl \mid u_2),\]
through the following composite. 
\begin{multline*}
          m_{u_2\kappa_1} \co \Sigma^{2p+2} V(1) \wdg \logthhku \simeq V(1) \wdg \logthhku \wdg \Sigma^{2p+2} \sph^{\triv} \xrightarrow{id \wdg u_2\kappa_1}\\ 
       \vo \wdg \logthhku \wdg \vo \wdg \logthhku \to V(1) \wdg \logthhku
\end{multline*}
Since $V(1) \wdg \thh(ku_p) \to \vo \wdg \thh(\kupl \mid u_2)$ is a map of monoids in the homotopy category of $S^1$-equivariant spectra, the following canonical diagram of $S^1$-equivariant spectra commutes up to homotopy. 
\begin{equation}\label{diag multlication commute with the thh and log thh}
    \begin{tikzcd}
        \Sigma^{2p+2}\vo \wdg \thh(\kup) \ar[d,"m_b"]\ar[r] & \ar[d,"m_{u_2\kappa_1}"]\Sigma^{2p+2} \vo \wdg \logthhku \\
        \vo \wdg \thh(\kup) \ar[r] &\vo \wdg  \logthhku
    \end{tikzcd}
\end{equation}
\end{cons}

\begin{prop}\label{prop thh to log thh is an equiv after localization for ell}
The map $\thh(\ell) \to \logthhell$ induces an equivalence 
\[L_{T(2)}\thh(\ell)^{hS^1} \xrightarrow{\simeq} L_{T(2)} \logthhell^{hS^1}.\]
\end{prop}
\begin{proof}
 There is an $E_\infty$-map $\kth(\zpl) \to \thh(\zpl)^{hS^1}$ and we have  $L_{T(2)}\kth(\zpl) \simeq 0$ due to \cite[Purity Theorem]{land2020purity}. This implies that \[L_{T(2)} \thh(\zpl)^{hS^1} \simeq 0.\] Since the cofiber of the map $\thh(\ell)^{hS^1} \to \logthhell^{hS^1}$ is given by $\Sigma \thh(\zpl)^{hS^1}$, this provides the desired result.
\end{proof}
\begin{rema}\label{rema ttwo wdg e is given by ttwo localization}
As mentioned earlier, the map $V(1) \to T(2)$ is given by the $T(2)$-localization \[V(1) \to L_{T(2)} V(1)\simeq T(2).\] For a given spectrum $E$, $T(2) \wdg E$ is a homotopy $T(2)$-module and therefore, $T(2) \wdg E$ is $T(2)$-local. Furthermore, $V(1) \wdg E \to T(2) \wdg E$ is a $T(2)$-equivalence as $V(1) \to T(2)$ is. Therefore, $\vo \wdg E \to T(2) \wdg E$ is given by the $T(2)$-localization:
\[V(1) \wdg E \to L_{T(2)} (V(1) \wdg E) \simeq T(2) \wdg E. \]
\end{rema}
\begin{prop}\label{prop log thh multiplication map is a local equivalence}
For the composite $S^1$-equivariant map: 
\begin{multline*}
    f \co\Sigma^{2p+2} \vo \wdg  \logthhell \to \Sigma^{2p+2}\vo \wdg \logthhku \xrightarrow{m_{u_2\kappa_1}}\\ \vo \wdg \logthhku \to \vo \wdg \thh(\kup)_1,
\end{multline*}
 $L_{T(2)}(f^{hS^1})$ is an equivalence. Here, the first and the last maps are those provided by Proposition \ref{prop weight grading on log thh ku}; indeed, the last map above is the projection to the factor of $\vo \wdg \logthhku$ corresponding to $1 \in \z/(p-1)$.
\end{prop}
\begin{proof}
Due to Theorem \ref{theo log thh of ku and ell} and Proposition \ref{prop subgroups of the weight grading on log thh ku}, $\pis f$ can be given by the composite map  
\begin{multline*}
     \pis f\co E(\lambda_1,d\log v_1) \otimes P(\kappa_1) \to \{1\} \otimes E(\lambda_1,d\log u_2) \otimes P(\kappa_1) \\\xrightarrow{\cdot u_2\kappa_1} \{u_2\} \otimes E(\lambda_1,d\log u_2) \otimes P(\kappa_1)
\end{multline*}
where the first map above sends $d\log v_1$ to $- d \log u_2$ and fixes the other generators and the second map multiplies by $u_2 \kappa_1$. The first map is an isomorphism and the second map above is an isomorphism in sufficiently large degrees. We deduce that $\pis f$ is an isomorphism in sufficiently large degrees. 

In particular, the cofiber of $f$, lets call it $C$, is bounded from above in homotopy. Therefore, $C^{hS^1}$ is also bounded from above in homotopy since $(-)^{hS^1}$ preserves coconnectivity. In particular, $L_{T(2)}(C^{hS^1}) \simeq 0$. This means that $L_{T(2)}(f^{hS^1})$ is an equivalence as desired.
\end{proof}

\begin{rema}
In the construction of the map $f$ above, if we used  $\vo \wdg \thh(ku_p)$ (together with its weight splitting and $m_{b}$) instead of $\vo \wdg \thh(\kupl \mid u_2)$, the proof above would fail to go through. This is because the cofiber of $f$ would not be bounded from above. This is precisely the reason why we use logarithmic THH for our computations. 
\end{rema}

The following is the non-logarithmic analogue of the the proposition above. 
\begin{prop}\label{prop mult by b on thhis a fixed point equivalence after telescopic localization }
For the composite $S^1$-equivariant map: 
\begin{multline*}
    g \co \Sigma^{2p+2} \vo \wdg    \thh(ku_p)_0 \to \Sigma^{2p+2} \vo \wdg \thh(ku_p) \xrightarrow{m_{b}}\\ \vo \wdg \thh(\kup) \to \vo \wdg \thh(\kup)_1,
\end{multline*}
$L_{T(2)}(g^{hS^1})$ is an equivalence. Here, the first  and the last maps are given by the $p-1$-grading on $\thh(ku_p)$. 
\end{prop}
\begin{proof}
For this, we consider the following (up to homotopy) commuting diagram of $S^1$-equivariant spectra. 
\begin{equation*}
    \begin{tikzcd}
        \Sigma^{2p+2}\vo \wdg  \thh(ku_p)_0 \ar[d] \ar[r]\ar[ddd,"g",shift right = 10,bend right = 60,swap]& \Sigma^{2p+2}  \vo \wdg  \logthhell\ar[d]\ar[ddd,"f",shift left = 15,bend left = 60]\\
       \Sigma^{2p+2} \vo \wdg \thh(ku_p) \ar[d,"m_b"] \ar[r]& \Sigma^{2p+2} \vo \wdg \logthhku\ar[d,"m_{u_2\kappa_1}"] \\
        \vo \wdg \thh(\kup) \ar[r]\ar[d]& \vo \wdg \logthhku  \ar[d]\\
        \vo \wdg \thh(\kup)_1 \ar[r,"id"]&  \vo \wdg \thh(\kup)_1 \\
    \end{tikzcd}
\end{equation*}
 The sequence of vertical maps on the  left hand side is the composite defining  $g$  and the sequence of vertical maps on the right hand side is the composite defining  the map $f$ in Proposition \ref{prop log thh multiplication map is a local equivalence}. The upper horizontal map is given by the passage from THH to log THH by noting $\thh(ku_p)_0 \simeq \thh(\ell_p)$ (see Proposition \ref{prop weight zero inclusion for thh ku}). The lower horizontal map is the identity map. The inner square above is given by Diagram \eqref{diag multlication commute with the thh and log thh} which commutes up to homotopy. By the definition of the rest of the maps, one observes that the diagram above commutes. 
 
 Due to Proposition \ref{prop log thh multiplication map is a local equivalence}, $L_{T(2)}(f^{hS^1})$ is an equivalence. Furthermore, the top horizontal map is also an equivalence after applying $L_{T(2)}(-^{hS^1})$ due to Propositions \ref{prop thh to log thh is an equiv after localization for ell} and \ref{prop weight zero inclusion for thh ku}. Since the lower horizontal map is also an equivalence, we deduce that $L_{T(2)}(g^{hS^1})$ is an equivalence as desired. 
\end{proof}

For the rest, we also let $b \in \vos\thh(ku_p)^{hS^1}$ denote the image of the higher Bott element $b \in \vos \kth(ku_p)$ under the trace map $\vos \kth(\kup) \to \vos\thh(ku_p)^{hS^1}$.

\begin{corr}\label{corr b carries weight 0 to weight 1 isomorphically in fixed points}
After restricting and corestricting, multiplication by $b$ provides an isomorphism
\[\cdot b \co T(2)_* \Sigma^{2p+2}\thh(ku_p)^{hS^1}_0 \xrightarrow{\cong} T(2)_* \thh(ku_p)^{hS^1}_1 \]
between the sets of weight $0$ and weight $1$ classes in $ T(2)_* \thh(ku_p)^{hS^1}$. 
\end{corr}
\begin{proof}
By Remark \ref{rema ttwo wdg e is given by ttwo localization} and the lax monoidal structure of the fixed points functor $-^{hS^1}$, this map is given by $\pis L_{T(2)} (g^{hS^1})$ which is an isomorphism due to Proposition \ref{prop mult by b on thhis a fixed point equivalence after telescopic localization }.
\end{proof} 

\begin{corr}\label{corr b is a unit in fixed points}
In $T(2)_*\thh(ku_p)^{hS^1}$, we have $b^{p-1} = -v_2$. In particular, $b \in T(2)_*\thh(ku_p)^{hS^1}$ is a unit.
\end{corr}
\begin{proof}
By Proposition \ref{prop properties of the higher Bott element b},  we have $b(b^{p-1}+v_2) = 0$ in  $T(2)_* \kth(ku_p)$. Using the ring map $\ttws \kth(ku_p) \to \ttws \thh(ku_p)^{hS^1}$,  we we obtain that 
\begin{equation}\label{eq bp is bv2 in negative tc}
b(b^{p-1}+v_2) = 0
\end{equation}
in $T(2)_*\thh(ku_p)^{hS^1}$.

Due to Proposition \ref{prop higher bott is homogenous of weight 1}, $b$ is  of weight $1$ in $\vos \thh(ku_p)^{hS^1}$. In particular, $b^{p-1}+v_2$ is of weight $0$  as $v_2$ is of weight $0$. However, multiplication by $b$ does not annihilate any non-trivial weight $0$ classes in $T(2)_*\thh(ku_p)^{hS^1}$ due to Corollary \ref{corr b carries weight 0 to weight 1 isomorphically in fixed points}. This, together with \eqref{eq bp is bv2 in negative tc} implies that $b^{p-1}+v_2= 0$ in $T(2)_*\thh(ku_p)^{hS^1}$ as desired.
\end{proof}

We are going to use the following two propositions to deduce that $L_{T(2)}\tc(m_b)$ is an equivalence; i.e.\ that $b$ is a unit in $\ttws \tc(ku_p)$.
\begin{prop}\label{prop mb fixed point is an equivalence}
The map \[L_{T(2)}(m_b^{hS^1})\co \Sigma^{2p+2}T(2) \wdg \thh(ku_p)^{hS^1} \xrightarrow{\simeq} T(2) \wdg \thh(ku_p)^{hS^1} \] 
is an equivalence.
\end{prop}
\begin{proof}
Using the lax structure of the homotopy fixed points functor $(-)^{hS^1}$ and Remark \ref{rema ttwo wdg e is given by ttwo localization}, one observes that the map 
$\pis L_{T(2)}(m_b^{hS^1})$
is precisely the map 
\[\ttws \thh(ku_p)^{hS^1} \to \ttws \thh(ku_p)^{hS^1}\]
given by multiplication by $b$. This is an isomorphism due to Corollary \ref{corr b is a unit in fixed points}.
\end{proof}

\begin{prop}\label{prop mb tate construction is an equivalence}
The map $L_{T(2)}(m_b^{tS^1})$ is an equivalence. 
\end{prop}
\begin{proof}
Since $m_b$ is an $S^1$-equivariant map, we have the following commuting diagram given by the canonical natural transformation in \cite[Corollary \RomanNumeralCaps{1}.4.3]{nikolausscholze2018topologicalcyclic}. 
\begin{equation*}
    \begin{tikzcd}
      \Sigma^{2p+2}T(2) \wdg \thh(ku_p)^{hS^1} \ar[r,"can"]\ar[r,swap,"\simeq"]\ar[d,"L_{T(2)}(m_b^{hS^1})"]\ar[d,swap,"\simeq"]& \Sigma^{2p+2} T(2) \wdg \thh(ku_p)^{tS^1} \ar[d,"L_{T(2)}(m_b^{tS^1})"]\\
    T(2) \wdg \thh(ku_p)^{hS^1} \ar[r,"can"]\ar[r,swap,"\simeq"]& T(2) \wdg \thh(ku_p)^{tS^1} 
    \end{tikzcd}
\end{equation*}

The maps $can$ above are equivalences since their fibers are given by  \[T(2) \wdg \Sigma \thh(ku_p)_{hS^1}\simeq 0\] due to \cite[Corollary \RomanNumeralCaps{1}.4.3]{nikolausscholze2018topologicalcyclic}.   The left hand vertical map is an equivalence due to Proposition \ref{prop mb fixed point is an equivalence}; and therefore, the right hand vertical map is also an equivalence.

\end{proof}

\begin{prop}\label{prop b is a unit in tc kup}
The higher Bott element $b \in T(2)_{2p+2}\tc(ku_p)$ is a unit.
\end{prop}
\begin{proof}
Recall that $m_b$ is a map of cyclotomic spectra by construction. Furthermore, the map 
\[\pis L_{T(2)} \tc(m_b) \co \ttws \Sigma^{2p+2}\tc(ku_p) \to \ttws \tc(ku_p)\]
is given by  multiplication by $b$. Therefore, it is sufficient to show that $L_{T(2)} \tc(m_b)$ is an equivalence. Since $m_b$ is cyclotomic, this induces a map of fiber sequences as follows, see \cite[Lemma \RomanNumeralCaps{2}.4.2]{nikolausscholze2018topologicalcyclic}.
\begin{equation*}
    \begin{tikzcd}[column sep = huge]
            \Sigma^{2p+2} T(2) \wdg \tc(ku_p) \ar[r,"L_{T(2)}\tc(m_b)"]\ar[d] &T(2) \wdg \tc(ku_p)\ar[d] \\
          \Sigma^{2p+2}T(2) \wdg \thh(ku_p)^{hS^1} \ar[d,"\varphi_p^{hS^1} - can"]\ar[r,"L_{T(2)}(m_b^{hS^1})"] \ar[r,swap,"\simeq"]& T(2) \wdg \thh(ku_p)^{hS^1} \ar[d,"\varphi_p^{hS^1} - can"]\\
          \Sigma^{2p+2} T(2) \wdg \thh(ku_p)^{tS^1} \ar[r,"L_{T(2)}(m_b^{tS^1})"] \ar[r,swap,"\simeq"] & T(2) \wdg \thh(ku_p)^{tS^1}
    \end{tikzcd}
\end{equation*}
The middle and the bottom horizontal maps are equivalences due to Propositions \ref{prop mb fixed point is an equivalence} and \ref{prop mb tate construction is an equivalence}. Since this is a map of fiber sequences, we deduce that the top horizontal map is an equivalence as desired. 
\end{proof}

\begin{theo}[Theorem \ref{theo Ausoni algebraic k theory of ku}]\label{theo restated Ausoni algebraic k theory of ku}
Let $p>3$ be a prime. There is an isomorphism of $\fp[b]$-algebras:
\[\ttws \kth(ku) \cong \ttws \kth(\ell) [b]/(b^{p-1}+v_2),\]
where $\lv b \rv = 2p+2$.

\end{theo}
\begin{proof}
Due to \cite[Purity Theorem]{land2020purity} and the Dundas-Goodwillie-McCarthy theorem, we have \[T(2)_*\kth(ku) \cong T(2)_*\tc(ku_p) \textup{\ \ and\ \ } T(2)_*\kth(\ell) \cong T(2)_*\tc(\ell_p).\]
Therefore, it is sufficient to prove the corresponding isomorphism of $\fp[b]$-algebras:
\begin{equation*}\label{eq quotient of poly to tc kup}
     h \co \ttws \tc(\ell_p) [b]/(b^{p-1}+v_2) \xrightarrow{\cong} \ttws \tc(ku_p),
\end{equation*}
at the level of topological cyclic homology.  Due to Proposition \ref{prop properties of the higher Bott element b}, $b(b^{p-1}+v_2) = 0$ in $T(2)_*\tc(ku_p)$. Since $b$ is a unit, we deduce that $b^{p-1} = -v_2$ in $T(2)_*\tc(ku_p)$. This provides the map $h$. 
Since $\ttws \tc(ku_p)$ is a $p-1$-graded ring with a unit $b$ in weight $1$ (Propositions \ref{prop higher bott is homogenous of weight 1} and \ref{prop b is a unit in tc kup}), it is periodic in its weight direction. In other words, multiplication by $b^i$ provides an isomorphism 
\[\cdot b^i \co \ttws\Sigma^{(2p+2)i}\tc(ku_p)_0 \xrightarrow{\cong} \ttws \tc(ku_p)_i\]
for each $0< i<p-1$. Furthermore, $\ttws \tc(ku_p)_0 \cong \ttws \tc(\ell_p)$ due to Proposition \ref{prop weight zero inclusion for thh ku}. This proves that $h$ is an  isomorphism as desired. 
\end{proof}

\begin{rema}\label{rema lax monoidal comparison of genuine and naive TC}

In the proof above, we used a lax symmetric monoidal transformation $\kth(-) \to \tc(-)$ given by the cyclotomic trace map. This follows by Blumberg, Gepner and Tabuada   \cite{blumberg2014uniqueness} if one uses the older definition of the lax symmetric monoidal functor $\tc(-)$. On the other hand, we defined the lax symmetric monoidal structure on $\tc(-)$ using cyclotomic spectra as in \cite{nikolausscholze2018topologicalcyclic} (Definition \ref{defi of tc and trivial cyclotomic spectrum functor}). To obtain the lax symmetric monoidal trace map for $\tc$ as in Definition \ref{defi of tc and trivial cyclotomic spectrum functor}, we only need to show that the forgetful functor from genuine cyclotomic spectra to cyclotomic spectra (in \cite[Section \RomanNumeralCaps{2}.6]{nikolausscholze2018topologicalcyclic}) is symmetric monoidal. By \cite[Construction IV.2.1]{nikolausscholze2018topologicalcyclic}, a lax symmetric monoidal functor into cyclotomic spectra amounts to a symmetric monoidal functor $F$ into $\sp^{BS^1}$ and a lax symmetric monoidal natural transformation   $F\to F^{tC_p}$. The functor from genuine cyclotomic spectra to $\sp^{BS^1}$ is given by the underlying Borel $S^1$-spectrum, which is a symmetric monoidal functor. Since the natural transformation between $C_p$-geometric fixed points and the $C_p$-Tate construction is lax symmetric monoidal, this shows that the forgetful functor to  cyclotomic spectra is symmetric monoidal as desired. 

\end{rema}

\subsection{Algebraic $K$-theory of the $2$-periodic Morava $K$-theory}\label{subsec kthry of mrv kthry proved and the fraction field}
At this point, Theorem \ref{theo kth of ku mod p} follows easily from our previous arguments. 

\begin{theo}[Theorem \ref{theo kth of ku mod p}]\label{theo restated kth of ku mod p}
Let $p>3$ be a prime. There is an isomorphism of $\fp[b]$-modules:
\[T(2)_*\kth(ku/p) \cong T(2)_* \kth(\ell/p)\otimes_{\fp[v_2]} \fp[b]\]
with $\lv b \rv = 2p+2$ and in the tensor product above, we take $v_2 = -b^{p-1}$. 
\end{theo}
\begin{proof}
As before, it is sufficient to  prove the same identity at the level of topological cyclic homology, i.e.\ we need to show that 
\[T(2)_*\tc(ku/p) \cong T(2)_* \tc(\ell/p)\otimes_{\fp[v_2]} \fp[b].\]

Recall from Construction \ref{cons tc of kup and ku mod p are graded} that  $\tc(ku/p)$ is a module over $\tc(ku_p)$ in  $p-1$-graded spectra. By Proposition \ref{prop weight zero inclusion of thh two periodic morava kthry}, we have 
\[T(2)_*\tc(ku/p)_0 \cong T(2)_*\tc(\ell/p).\]
Furthermore, there is a unit $b \in T(2)_*\tc(ku_p)$ of weight $1$, (Propositions \ref{prop higher bott is homogenous of weight 1} and \ref{prop b is a unit in tc kup}). Therefore, multiplying by powers of $b$ induces isomorphisms
\[\cdot b^i \co T(2)_*\Sigma^{(2p+2)i}\tc(ku/p)_0 \xrightarrow{\cong} T(2)_*\tc(ku/p)_i\]
for each $0<i<p-1$. This provides the desired result. 

\end{proof}

Now we prove Theorem \ref{theo k theory of the fraction field of ku} verifying the the conjectural formula of Ausoni and Rognes \cite[Section 3]{ausoni2009kthryoffractionfield} that we stated in \eqref{eq conjectural formula of ausoni rognes}. For this, note that due to \cite[Purity Theorem]{land2020purity}, $T(2) \wdg \kth(\ff (ku_p))$ and $T(2) \wdg \kth(\ff(\ell_p))$ are given by the cofibers of the transfer maps $T(2) \wdg \kth(ku/p) \to T(2) \wdg \kth(ku_p)$ and  $T(2) \wdg \kth(\ell/p) \to T(2) \wdg \kth(\ell_p)$ respectively, see \cite[Diagrams 3.1 and 3.10]{ausoni2009kthryoffractionfield}. 

\begin{theo}\label{theo k theory of the fraction field of ku}
Let $p>3$ be a prime. There is an isomorphism of $\fp[b]$-modules:
\begin{equation}\label{eq ttwo homology of fraction field of ku}
T(2)_* \kth(\ff (ku_p))  \cong T(2)_*\kth(\mathit{ff}(\ell_p))\otimes_{\fp[v_2]} \fp[b]
\end{equation}
where $v_2 = -b^{p-1}$ and $\lv b \rv = 2p+2$.

\end{theo}
\begin{proof}
The trace map $\kth(-) \to \tc(-)$ is a $T(2)$-equivalence for $ku_p$, $ku/p$, $\ell_p$ and $\ell/p$ \cite[Corollary E]{land2020purity}. Therefore, we obtain that $T(2) \wdg \kth(ku_p)$ is a monoid in the homotopy category of $p-1$-graded spectra and $T(2) \wdg \kth(ku/p)$ is a  left module over $T(2) \wdg \kth(ku_p)$ in the homotopy category of $p-1$-graded spectra. 

Let 
\[\tau \co T(2) \wdg\kth(ku/p) \to T(2) \wdg \kth(ku_p)\]
denote the map induced by transfer along $ku_p \to ku/p$. Since 
\[T(2) \wdg \kth(ku/p) \simeq \bigvee_{0\leq i<p-1 } T(2) \wdg \kth(ku/p)_i,\]
it is sufficient to understand the restriction of $\tau$ to $T(2) \wdg \kth(ku/p)_i$ for each $i$. For $i=0$, we consider the commuting diagram of $E_1$-rings:
\begin{equation*}
 \begin{tikzcd}
     \ell_p \ar[r]\ar[d] & \ell/p\ar[d]\\
     ku_p \ar[r] & ku/p.
 \end{tikzcd}
\end{equation*}
 We obtain a commuting diagram of spectra:
\begin{equation}\label{diag transfer and extension commute}
 \begin{tikzcd}
     \kth(\ell/p) \ar[r]\ar[d] & \kth(\ell_p)\ar[d]\\
     \kth(ku/p) \ar[r] & \kth(ku_p)
 \end{tikzcd}
\end{equation}
by using the following equivalence of the corresponding functors induced at the level of module categories:
\[ku_p \wdg_{\ell_p} -  \simeq   ( ku_p \wdg_{\ell_p} \ell/p) \wdg_{\ell/p} - \simeq ku/p \wdg_{\ell/p} -.\]

Let $\tau' \co T(2) \wdg \kth(\ell/p) \to T(2) \wdg \kth(\ell)$ denote the map induced by transfer along $\ell \to \ell/p$. Diagram \eqref{diag transfer and extension commute} provides that the restriction of $\tau$ to  $T(2) \wdg \kth(ku/p)_0$ is given by the following map. 
\begin{multline*}
T(2) \wdg \kth(ku/p)_0 \simeq T(2) \wdg  \kth(\ell/p) \xrightarrow{\tau'}\\ T(2) \wdg\kth(\ell_p) \simeq T(2) \wdg\kth(ku_p)_0 \to T(2) \wdg \kth(ku_p) \end{multline*}
Here, the last map is the inclusion of the weight $0$-component and the equivalences above are provided by Propositions \ref{prop weight zero inclusion for thh ku} and \ref{prop weight zero inclusion of thh two periodic morava kthry}. Let $\tau_0$ denote the map \[\tau_0 \co T(2) \wdg \kth(ku/p)_0 \to T(2) \wdg\kth(ku_p)_0\]in the composite above.

Let $0<i<p-1$. To describe the restriction of $\tau$ to $T(2) \wdg \kth(ku/p)_i$, we use the fact that $\tau$ is a map of $T(2) \wdg \kth(ku_p)$-modules in the stable homotopy category, see \cite[Section 3]{ausoni2009kthryoffractionfield}. By Propositions  \ref{prop higher bott is homogenous of weight 1} and \ref{prop b is a unit in tc kup},  there is a unit $b^i \in T(2)_* \kth(ku_p)$ of weight $i$. Omitting the  suspension functor,  let $m_1 \co T(2) \wdg \kth(ku/p)_0 \to T(2) \wdg \kth(ku/p)_i$ and $m_2 \co T(2) \wdg \kth(ku_p)_0 \to T(2) \wdg \kth(ku_p)_i$ denote the equivalences given by  multiplication with $b^i \in T(2)_* \kth(ku_p)$. Abusing notation, let $m_1$ and $m_2$ also denote the respective endomorphisms of $T(2) \wdg \kth(ku/p)$ and $T(2) \wdg \kth(ku_p)$.  We have the following up-to homotopy  commuting diagram of spectra. 
\begin{equation*}
    \begin{tikzcd}[ column sep=scriptsize]
    & & T(2) \wdg \kth(ku_p)_0 \ar[d] \ar[dr,"m_2"]\ar[dr,"\simeq",swap]\\
        T(2) \wdg \kth(ku/p)_0 \ar[r] \ar[urr,"\tau_0"] \ar[d,"m_1"]\ar[d,"\simeq",swap] & \ar[d,"m_1"]T(2) \wdg \kth(ku/p) \ar[r,"\tau"] & T(2) \wdg \kth(ku_p) \ar[d,"m_2"] & T(2) \wdg \kth(ku_p)_i \ar[ld] \\
        T(2) \wdg \kth(ku/p)_i \ar[r]  &T(2) \wdg \kth(ku/p) \ar[r,"\tau"]  & T(2) \wdg \kth(ku_p) &  
    \end{tikzcd}
\end{equation*}
Here, the unmarked arrows are the canonical inclusions. The bottom left hand square and the right hand diagram commute since $T(2) \wdg \kth(ku/p)$ and $T(2) \wdg \kth(ku_p)$ are modules over $T(2) \wdg \kth(ku_p)$ in the homotopy category of $p-1$-graded spectra. The inner square commutes as $\tau$ is a map of modules over $T(2) \wdg \kth(ku_p)$ in the homotopy category of spectra. The top left diagram commutes due to our previous identification of $\tau_0$.

The commuting diagram above shows that the restriction of $\tau$ to $T(2) \wdg \kth(ku/p)_i$ is given by the composite: 
\begin{multline*}
T(2) \wdg \kth(ku/p)_i\xrightarrow[\simeq]{m_1^{-1}} T(2) \wdg \kth(ku/p)_0 \xrightarrow{\tau_0} \\T(2) \wdg \kth(ku_p)_0 \xrightarrow[\simeq]{m_2} T(2) \wdg \kth(ku_p)_i \to T(2) \wdg \kth(ku_p).
\end{multline*}
Letting $\tau_i \co T(2) \wdg \kth(ku/p)_i \to T(2) \wdg \kth(ku_p)_i$ be  as in the composite above, we obtain that $\tau$ is given by 
\[\tau \simeq \bigvee_{0\leq i<p-1}\tau_i\]
and that each $\tau_i$ is equivalent to $\tau_0$ up to a suspension. Since $\tau_0$ is equivalent to $\tau'$, this provides the desired splitting of the cofiber of $\tau$ as a coproduct of shifted copies of $T(2) \wdg \kth(\ff(\ell_p))$. This proves \eqref{eq ttwo homology of fraction field of ku} as an isomorphism of abelian groups. Due to the argument above, the resulting cofactors of $T(2) \wdg \kth(ku/p)$ are connected through multiplication by $b$ and this shows that \eqref{eq ttwo homology of fraction field of ku} is an isomorphism of $\fp[b]$-modules.

\end{proof}

\appendix

\section{Graded cyclotomic spectra} \label{appendix graded tc}

For this  section, let $m$ be a positive integer such that $m \mid p-1$. Here, our goal is to construct a symmetric monoidal  functor \[\thh \co \alg_{E_1}(\fun(\z/m,\sp)) \to  \fun(\z/m,\cycsp)\]
and show that the resulting diagram: 
\begin{equation}\label{diag main diag of the appendix}
\begin{tikzcd}
    \alg_{E_1}(\fun(\z/m,\sp)) \ar[d] \ar[r,"\thh"] & \fun(\z/m,\cycsp) \ar[r,"\tc"] \ar[d,"D"] & \fun(\z/m,\sp)\ar[d,"D'"]\\
    \alg_{E_1}(\sp) \ar[r,"\thh"] & \cycsp \ar[r,"\tc"] & \sp
    \end{tikzcd}\tag{\ref{diag for graded thh and graded tc}}
\end{equation}
of lax symmetric monoidal functors commutes. Note that this diagram is also stated as Diagram \eqref{diag for graded thh and graded tc} in Section \ref{sec graded THH, TC, neg TC}. Recall that the vertical functors above are given by left Kan extending through $\z/m\to 0$, i.e.\ they provide the corresponding underlying objects. Furthermore, the upper right hand horizontal arrow $\tc$ is given by levelwise application of $\tc \co \cycsp \to \sp$.

We first prove the following proposition which states that the right hand square in Diagram \eqref{diag for graded thh and graded tc} commutes. 

\begin{prop}\label{prop graded TC}
Let $m>0$ be a positive integer such that $m\mid p-1$. Then the following diagram:
\begin{equation*}
\begin{tikzcd}
\fun(\z/m,\cycsp)\ar[d,"D"] \ar[r,"\tc"]& \ar[d,"D'"]\fun(\z/m,\sp)\\
 \ar[r,"\tc"]\cycsp & \sp,
\end{tikzcd}
\end{equation*}
of lax symmetric monoidal functors commutes. In other words, the right hand side of Diagram \eqref{diag for graded thh and graded tc}, commutes.
\end{prop}
\begin{proof}
Let $R$ and $R'$ denote the right adjoints of $D$ and $D'$ respectively. The functors $R$ and $R'$ are given by restriction along $\z/m \to 0$. Here, we denote the top horizontal arrow by $\tc^{\textup{level}}$ to distinguish it from the bottom horizontal arrow $\tc$. 

First, we show that there is a lax symmetric monoidal natural transformation 
\[\phi \co D'\tc^{\textup{level}} \to \tc D,\]
later, we complete the proof by showing that $\phi$ is an equivalence. By adjunction, it is sufficient to obtain a lax symmetric monoidal transformation
\[\tc^{\textup{level}} \to R'\tc D.\]
Since precomposition followed by postcomposition agrees with postcomposition followed by precomposition, we have $R' \tc \simeq \tc^{\textup{level}}R$. Therefore, it is sufficient to obtain a lax symmetric monoidal transformation
\[\tc^{level} \to \tc^{level}RD\]
and this is given by the unit of the adjunction $D\dashv R$. This provides  $\phi$ above. Indeed, $\phi$ is given by the canonical map 
\[\phi_X \co \vee_{i \in \z/m} \tc(X_i) \to \tc(\vee_{i \in \z/m} X_i).\]
Since $m \neq 0$, the coproducts above are finite and due to \cite[Corollary \RomanNumeralCaps{2}.1.7]{nikolausscholze2018topologicalcyclic}, colimits of cyclotomic spectra agree with those of the underlying spectra. Furthermore, both $\cycsp$ and $\sp$ are stable and therefore these coproducts are the corresponding products. As $\tc$ commutes with finite products, we obtain that $\phi$ is an equivalence as desired.
\end{proof}

What remains is to  construct the left hand square in Diagram \ref{diag for graded thh and graded tc} and show that it commutes. 

For the rest, we let $\grm(\C)$ denote $\fun(\z/m,\C)$ for a given presentably symmetric monoidal $\infty$-category $\C$. 
Abusing notation, let  
\[(-)^{tC_p}\co \grm(\sp^{BS^1})\to \grm(\sp^{BS^1})\]
also denote the lax symmetric monoidal functor given by levelwise application of $(-)^{tC_p}$. Slightly diverting from the notation of \cite{nikolausscholze2018topologicalcyclic}, we let  $\textup{Leq}\big(\grm(\sp^{BS^1}),(-)^{tC_p}\big)$ denote the $\infty$-category defined as the lax equalizer of the identity endofunctor and the endofunctor $(-)^{tC_p}$ on $\grm(\spcircle)$ in the sense of \cite[Definition \RomanNumeralCaps{2}.1.4]{nikolausscholze2018topologicalcyclic}. The $\infty$-category $\leqgrsp$ is defined via the following pullback square. 

\begin{equation*}
    \begin{tikzcd}
        \leqgrsp \ar[r]\ar[d]& \grm(\spcircle)^{\Delta^1}\ar[d,"\textup{ev}_0\times \textup{ev}_1"]\\
        \grm(\spcircle) \ar[r,"{(id,(-)^{tC_p})}"] & \grm(\spcircle) \times \grm(\spcircle)
    \end{tikzcd}
\end{equation*}

 In particular, the objects of this pullback $\infty$-category are given by an object of $E \in \grm(\spcircle)$ and a morphism $E \to E^{tC_p}$.


In  \cite[Appendix A]{antieau2020beilinson}, the authors construct $\thh$ as a functor on graded ring spectra and show that it fits into the following commuting diagram of  symmetric monoidal functors \cite[Proposition A.5 and Corollary A.15]{antieau2020beilinson}.
\begin{equation} \label{diag Appendix A diagram}
\begin{tikzcd}
    \alg_{E_1}(\grm(\sp)) \ar[d]\ar[r,"\thh"] & \leqgrsp \ar[d,"D''"]\\
    \alg_{E_1}(\sp)\ar[r,"\thh"] & \cycsp
    \end{tikzcd}
\end{equation}

As usual, the vertical arrows are induced by left Kan extension through $\z/m\to 0$. Here, we omit the functor $L_p$ (given by left Kan extension through $\cdot p \co \z/m \to \z/m$)  since $L_p$ is the identity functor whenever $p=1$ in $\z/m$.  

\begin{cons}
Let $ \Alg_{E_1}(\grmsp) \to \grm(\cycsp)$ be the composite of the upper horizontal arrow in Diagram \eqref{diag Appendix A diagram} with the equivalence provided in the proposition below. This provides the left hand upper horizontal arrow in Diagram \eqref{diag main diag of the appendix} and the commuting diagram in the following proposition, together with Diagram \eqref{diag Appendix A diagram}  ensures that the left hand square in Diagram \eqref{diag main diag of the appendix} commutes.
\end{cons}

What remains is to prove the following proposition.
\begin{prop}\label{prop graded cyclotomic is cyclotomic in graded}
Let $m>0$ such that $m \mid p-1$. There is an equivalence of symmetric monoidal $\infty$-categories:
\begin{equation}\label{eq functor from graded cyclotomic to cyclotomic in graded}\leqgrsp  \xrightarrow{\simeq} \grm(\cycsp)
\end{equation}
such that the following diagram commutes. 
\begin{equation*}
    \begin{tikzcd}
      \leqgrsp\ar[r,"\simeq"]\ar[d,"D''"]& \grm(\cycsp)  \ar[dl,"D"]\\
      \cycsp
    \end{tikzcd}
\end{equation*}
\end{prop}
\begin{proof}
We construct an equivalence in the opposite direction.
Due to \cite[Construction \RomanNumeralCaps{4}.2.1]{nikolausscholze2018topologicalcyclic},  giving a symmetric monoidal functor  \[T \co \grm(\cycsp) \to \leqgrsp\] is equivalent to giving a symmetric monoidal functor
\[F \co \grm(\cycsp) \to \grm(\spcircle)\]
together with a lax symmetric monoidal transformation $ F \to (-)^{tC_p} \circ F$. 

Applying \cite[Construction \RomanNumeralCaps{4}.2.1]{nikolausscholze2018topologicalcyclic} to the identity functor of $\cycsp$, one obtains a  symmetric monoidal  functor:
\[H\co \cycsp \to \sp^{BS^1},\]
together with a lax symmetric monoidal transformation $H \to (-)^{tC_p} \circ H$. By \cite[Corollary 3.7]{nikolaus2016stablemultyoneda}, this provides the desired  symmetric monoidal functor $F$ above. Furthermore, the lax symmetric monoidal transformation $H \to (-)^{tC_p} \circ H$ applied to the following lemma provides the desired lax symmetric monoidal transformation $ F \to (-)^{tC_p} \circ F$. This natural transformation and $F$ provides the symmetric monoidal functor $T$ above. Since $\fun(\z/m,-)$ commutes with limits, it commutes with the  pullback square defining $\cycsp$ as a lax equalizer; therefore, $T$ is an equivalence as desired. The functor claimed in \eqref{eq functor from graded cyclotomic to cyclotomic in graded} is now given by $T^{-1}$. 

For the second statement in the proposition, it is sufficient to show that the following diagram commutes.  
\begin{equation*}
\begin{tikzcd}
    \grm(\cycsp) \ar[r,"T"] \ar[d,"D"] & \leqgrsp\ar[dl,"D''"]\\
    \cycsp
\end{tikzcd}
\end{equation*}
Let $R$ and $R''$ denote the right adjoints of $D$ and $D''$ respectively. These are given by the corresponding restriction functors along $\z/m\to 0$. Since $T$ is an equivalence, $T^{-1}R''$ is a right adjoint to $D''T$ and therefore, it is sufficient to show that $R\simeq T^{-1}R''$, i.e.\ the right adjoints of $D$ and $D''T$ agree. For this, it is sufficient to show that $TR \simeq R''$. This  follows by the fact that the functor $F$ and the lax transformation $F\to (-)^{tC_p} \circ  F$ are defined levelwise. 
\end{proof}
\begin{lemm}
Let $\eta \co T\to S$ be a lax symmetric monoidal transformation of lax symmetric monoidal functors between presentably symmetric monoidal $\infty$-categories $\C$ and $\D$. Then applying $\eta$ levelwise induces a lax symmetric monoidal transformation between the induced lax symmetric monoidal functors from $\grm(\C)$ to $\grm(\D)$.
\end{lemm}
\begin{proof}
We follow closely \cite[Section 3]{nikolaus2016stablemultyoneda}. Since the $\infty$-category of lax symmetric monoidal functors is a full subcategory of the $\infty$-category of functors over $\textup{NFin}_*$, it is sufficient  to show that  $\eta$ provides a map $\Delta^1 \to \Map_{\textup{NFin}_*}(\grm(\C)^{\otimes}, \grm(\D)^\otimes)$ of simplicial sets where the vertices of $\Delta^1$ correspond to the lax symmetric monoidal functors induced by $T$ and $S$. Using the universal property defining $\text{hom}_{/\text{NFin}_*}(-,-)$ in \cite[Section 3]{nikolaus2016stablemultyoneda}, one obtains the second map  below: 
\[\Delta^1 \to \Map_{\textup{NFin}_*}(\C^\otimes,\D^\otimes) \to \Map_{\textup{NFin}_*}(\textup{hom}_{/\text{NFin}_*}(\z/m^\otimes,\C^{\otimes}),\textup{hom}_{/\text{NFin}_*}(\z/m^\otimes,\D^{\otimes})),\]
where the first map represents $\eta$.
 Using the definition of $\grm(-)$ as a full simplicial subset of $\textup{hom}_{/\text{NFin}_*}(\z/m^\otimes,-)$ and  \cite[Corollary 3.7]{nikolaus2016stablemultyoneda}, we deduce that the $1$-simplex in the composite above lies in $ \Map_{\textup{NFin}_*}(\grm(\C)^{\otimes}, \grm(\D)^\otimes)$ with vertices corresponding to the lax symmetric monoidal functors induced by $T$ and $S$ as desired.

\end{proof}

\small
\providecommand{\bysame}{\leavevmode\hbox to3em{\hrulefill}\thinspace}
\providecommand{\MR}{\relax\ifhmode\unskip\space\fi MR }
\providecommand{\MRhref}[2]{%
  \href{http://www.ams.org/mathscinet-getitem?mr=#1}{#2}
}
\providecommand{\href}[2]{#2}

\end{document}